\documentclass[11pt]{article}
\usepackage{amsmath, amsfonts, amsthm, amssymb}
\usepackage[]{pstricks,pst-node,pst-text,pst-3d}
\usepackage{graphics}

\usepackage{fullpage}

\definecolor{cstblue}{rgb}{0,0,0.85}
\definecolor{cstred}{rgb}{0,0,0}

\newcommand{\f}[2]{\frac{#1}{#2}}

\def\del{\partial}

\newcommand{\B}{\mathcal{B}}

\newcommand{\RR}{\mathbb{R}}
\newcommand{\eps}{\varepsilon}

\newcommand{\D}{\mathrm{D}}

\newcommand{\BBN}[1]{\left\| #1 \right\|}

\newcommand{\SP}{\hspace{1pt}}
\newcommand{\DSP}{\hspace{2pt}}

\newtheorem{theorem}{Theorem}[subsection]
\newtheorem{proposition}[theorem]{Proposition}
\newtheorem{definition}[theorem]{Definition}
\newtheorem{lemma}[theorem]{Lemma}
\newtheorem{remark}[theorem]{Remark}

\title{Stability and Convergence of Relaxation Schemes \\ to
      Hyperbolic Balance Laws via a Wave Operator\\ {\footnotesize
(In: {\it Journal of Hyperbolic Differential Equations (2015), 12-1, 189-219.})}}

\author{Alexey Miroshnikov\thanks {Department of Mathematics, University of Massachusetts, Amherst, MA 01003, USA} \;   and \, Konstantina Trivisa
\thanks{Department of Mathematics, University of Maryland, College Park, MD 20742, USA} }

\begin{document}

\numberwithin{equation}{section}

\date{}
\maketitle

\begin{center}
\textbf{Abstract}
\end{center}
\noindent
This article deals with relaxation approximations of nonlinear systems of hyperbolic balance laws.
We introduce a class of relaxation schemes and establish their stability and convergence  to the solution of hyperbolic balance laws before the formation of shocks, provided that we are within the framework of the compensated compactness method.
Our analysis treats systems of hyperbolic balance laws with source terms satisfying a special mechanism   which induces weak dissipation in the spirit of Dafermos \cite{Dafermos06}, as well as  hyperbolic balance laws with more general source terms.
The rate of convergence of the relaxation system to a  solution of the balance laws in the smooth regime is established.
Our work follows in spirit  the analysis presented in \cite{AMT-2005, JX-1995} for systems of hyperbolic conservation laws without source terms.


\section{Introduction}\label{S1}
Relaxation approximations of hyperbolic balance laws is of essence for the investigation of models arising in continuum mechanics and kinetic theory of gases, and serve as a ground stage for the design of numerical schemes for hyperbolic balance laws.
In this article we introduce a class of  relaxation schemes for the approximation of solutions to the hyperbolic balance law
\begin{equation}\label{BLAWMDIM}
\displaystyle \del_t u+ \sum_{j=1}^d \del_{x_j} F_j(u) = G(u), \quad u \in
\RR^n, \,\, (x,t)\in\RR^d \times [0,\infty)\\\\
\end{equation}
and address the issues of stability and convergence.

\par\smallskip

The class of relaxation schemes introduced in this work are of the form
\begin{equation}\label{RXSYSTMDIM}
\left\{
\begin{aligned}
\del_t u + \sum_{j=1}^d \del_{x_j} v_j &= 0 \\
\del_t v_i + A_i \del_{x_i} u &= -\frac{1}{\eps} \Bigl( v_i-F_i(u) +
\mathcal{R}_i(x,t)\Bigr), \quad i=1,\dots,d
\end{aligned} \right.
\end{equation}
with $v_i\in \RR^n$, $A_i$ symmetric, positive definite matrix, and
\begin{equation}\label{RFUNCDEFMDIM}
    \mathcal{R}_i(x,t) = \frac{1}{d}\int^{x_i} G(u(x_1,\dots,
    x_{i-1},z,x_{i+1},\dots,x_d,t))\,dz.
\end{equation}

\par\smallskip

Excluding $v_i$ for all $i=1,\dots,d$ from the equation \eqref{RXSYSTMDIM}$_1$ we obtain
\begin{equation}\label{RXBLAWMDIM}
\begin{aligned}
\del_t u + \sum_{j=1}^d \del_{x_j} F_j(u) = G(u) + \eps \Bigl(\sum_{j=1}^d A_j u_{{x_j}
{x_j}} - u_{tt}\Bigr)
\end{aligned}
\end{equation}
that approximates the system of balance laws \eqref{BLAWMDIM}.
For the relaxation approximations considered in this work the stabilization mechanism is the regularization by the {\em wave operator}  in \eqref{RXBLAWMDIM}.

The convergence properties of relaxation systems and associated relaxation schemes for scalar conservation laws are presently well understood (see \cite{AN-1996, CLL-1994, JX-1995, N-1996}). When the zero-relaxation limit is a system of conservation laws, the dissipative effect of relaxation is often subtle to capture,  yet there is a vast literature on convergence results in that direction (see \cite{AMT-2005,Serre-2000, GT-2001, T-1998, T-1999} and the references there rein).

By  contrast, the relaxation approximation of nonlinear systems of  hyperbolic balance laws presents
major additional challenges and it is the subject of current intense research investigation. It is well known that standard methods that solve correctly
systems of conservation laws can fail in solving systems of balance laws, especially when approaching equilibria or near to equilibria solutions. In addition,  standard approximating procedures produce often unstable methods when they are applied to
coupled systems of conservation or balance laws. Even at the theoretical level, the
presence of the production term (source term) in the equation  results to the
amplification in time of even small oscillations in the solution.
Due to these challenges special mechanisms that induce dissipation are often desirable and have been proven effective in obtaining long-term stability (cf. Dafermos \cite{Dafermos06}, \cite{Dafermos10}).

In the present article  we identify a class of relaxation schemes suitable
for the approximation of solutions to certain systems of hyperbolic balance laws arising
in continuum physics. The relaxation schemes proposed in our work provide a very
effective mechanism for the approximation of the solutions of these systems with a very
high degree of accuracy.

\par\smallskip

The main contribution of the present article to the existing theory can be characterized
as follows:
\begin{itemize}
\item A new class of relaxation schemes \eqref{RXSYSTMDIM} is  introduced.
The novelty of the  relaxation systems proposed in this work lies in the introduction of the {\em global term} $\{\mathcal{R}_i(x,t)\}$  in \eqref{RXSYSTMDIM}-\eqref{RFUNCDEFMDIM}. The presence of this term  in the relaxation system allows us to relax both the flux and the antiderivative of the source term simultaneously.
We refer the reader to the article \cite{KM-2003}, where a relevant idea was proposed  for the numerical treatment of shallow water equations.
The present article is the first step towards the construction of fully discrete  schemes and the development of  numerical methods  for the approximation of solutions  to complex nonlinear multidimensional systems of hyperbolic balance laws arising in applications.
Our analysis provides a rigorous
proof of the relaxation  limit and a rate of convergence before the formation of shocks.

\item A comparison is presented between the relaxation  system introduced here and  an alternative relaxation system for which the source term $G(u)$ appears in the right-hand side of the first equation (see  \eqref{RXSYSTMDIMALT} in Section \ref{S6}). Note, dealing with  \eqref{RXSYSTMDIMALT} one faces arduous challenges. More specifically, {the time derivative} of the source term appears in the energy functional posing enormous challenges in the analysis,  an additional hypothesis is required for the establishment of stability (see (H7), Section \ref{S6}), the issue of compactness is problematic.

\item The presence of a source term in our system requires us to modify the relative entropy method significantly. The  {\em modified relative method}  presented in this work  relies on a {\em relative potential} (Section \ref{S8}).
The introduction of this concept is required in order to deal with the source term $G$ which typically satisfies no growth conditions.
The relative potential assists in ``tracking" the contribution of the  source term; it becomes a part of a Lyapunov functional which monitors the evolution of the difference of the solutions and enables us to establish a convergence rate of  order ${\mathcal O}(\varepsilon^2).$ In the case of the general source $G$ the terms associated with the source  are treated as error (cf.\ Section 6.4).
\end{itemize}

The reader should contrast the relaxation system presented in this work with other relaxation systems proposed in the literature \cite{MT-2013a, N-1996}, where relaxation approximations to hyperbolic balance laws were proposed and rigorously established. In \cite{MT-2013a}  relaxation approximations are  constructed by the introduction of special variables the so-called {\em internal variables}. In that setting applications to physical systems in elasticity and combustion theory are presented. In the former case, relaxation is introduced  via {\em stress approximation}, whereas in the latter case  via {\em approximation of pressure}.
Reference \cite{N-1996}  treats the Cauchy problem for $2 \times 2$  semilinear and quasilinear hyperbolic systems with a singular relaxation term and  presents the convergence to equilibrium of the solutions of these problems as the singular perturbation parameter tends to $0.$

\par\smallskip

 In the center of our analysis lies the entropy structure of the balance law and the dissipative nature of the source term.
The main ingredients of our approach can be formulated as follows:
\begin{itemize}
\item The representation of the {\em global term} ${\mathcal R}$ in the formulation of the relaxation system  enables us to obtain the stability estimates in Section \ref{S3}. These estimates are used subsequently to establish the stability of the relaxation approximations and in fact justify the dissipative character of our   systems.

\item The entropy structure of the balance law provides the basis for the use of the compensated compactness method. Recall that   a pair of functions $\eta = \eta(u), q= q(u)$ are called the entropy-entropy flux pair if $(\eta, q)$ solve the linear hyperbolic system
$$ \D q = \D\eta \D F.$$
In Section \ref{S5}, we show that the relaxation approximations satisfy
$$\del_t \eta(u^{\eps}) + \del_x q(u^{\eps}) \subset \mbox{compact set of}\,\, H^{-1}_{loc}({\mathcal O})$$
{for a certain class of entropy-entropy flux pairs $\eta-q$, which is one of the main ingredients for the establishment of convergence results within the compensated compactness method (cf. Serre \cite{Serre-book}).}

\item
The Lyapunov functional construction in Section \ref{S8}  relies on the  relative entropy via a Chapman-Enskog-type expansion and is used to provide a simple and direct convergence framework  before formation of shocks as well as suitable error estimates.

\item A physically motivated dissipation mechanism
associated with the source term in \eqref{BLAWMDIM}. The concept of {\em weak dissipation} for hyperbolic balance laws was
introduced by Dafermos in \cite{Dafermos06}. The reader may contrast the result of Theorem \ref{WDESTTHM}  for weakly
dissipative source terms with Theorem \ref{gen-source} which  corresponds to the case of a
general source.
\end{itemize}

\par\smallskip

The outline of our article is as follows: In Section \ref{S2}  we present the basic notation and hypothesis.
In Section \ref{S3} we present the stability estimates which yield the stability of the relaxation systems. The compactness properties of the approximate solutions are discussed in Section \ref{S4}.
Section \ref{S5} is devoted to error estimates for smooth solutions via
the relative entropy method as well as proofs of convergence. The multidimensional case is treated in Section \ref{S7}. Section \ref{S8} presents applications in elasticity and combustion.

\section{Notation and Hypotheses}\label{S2}
For the convenience of the reader we collect in this section all the relevant notation and hypotheses.
Here and in what follows:
\begin{enumerate}
\item[] {\bf 1.} $G, R, F_i$, $i=1,\dots,d$ denote the mappings
$G, R,  F_{i}: \RR^n \to \RR^n$. In our presentation, $G(u), R(u), F_i(u)$ are treated as column vectors.

\item[] {\bf 2.} $\D$ denotes the differential with respect to the state vectors $u \in \RR^n$. When used in conjunction with matrix notation, $\D$ represents  a row operation:
\begin{equation*}
\D=[\del/\del u^1, \dots, \del/\del u^n ].
\end{equation*}
\end{enumerate}

\subsection{Entropy Structure}

Some additional assumptions on the system \eqref{BLAWMDIM} read:
\begin{itemize}
\item The system \eqref{BLAWMDIM} is equipped with a globally
defined entropy $\eta(u)$ and corresponding fluxes $q_{i}(u)$, $i=1,\dots,d$, such that
\begin{equation}\tag{H1}\label{RXENTPROP}
\begin{aligned}
& \eta:\RR^n \rightarrow \RR  \,\,\, \mbox{is strictly convex}, \\
& {\D}\eta(u) \SP {\D}F_i(u)={\D} q_i (u)\\
& \beta \SP {\bf I} \leq D^2 \eta (u) \leq \tfrac{1}{2} \alpha \SP {\bf I}, \quad u\in \RR^n, \,\, \mbox{for}\,\,  \alpha,\beta>0 \SP ,\\
& \eta(u) \ge \eta(0) =0, \SP \D\eta(0)=0.
\end{aligned}
\end{equation}
\end{itemize}
We recall that $2 \times 2$ as well as physical systems of hyperbolic conservation laws are always equipped with an entropy-entropy flux pair. The same holds true for symmetric hyperbolic systems \cite{Dafermos10}.

\par\smallskip

\subsection{Subcharacteristic condition}
The Whitham relaxation {\em subcharacteristic condition} presented below will be essential for the dissipativiness of our system \cite{AN-1996, Dafermos10, HN03, N-1996}.
\begin{itemize}

\item
(Case $d=1$. Systems with a strictly convect entropy).

For $\alpha>0$ in \eqref{RXENTPROP} there exists $\nu>0$ and symmetric, positive definite matrix $A$ such that
\begin{equation}\label{AHYPSCENT1D}
\tfrac{1}{2}\bigl( A\D^2{\eta}(u) + \D ^2{\eta}(u)A\bigr)-\alpha \SP
\D F(u)^{\top} \D F(u)\geq\nu {\bf I}\,, \quad u\in \RR^n. \tag{H2}
\end{equation}
The positivity of this  term is required for dissipativiness in Lemma \ref{energy}.

\item (General case $d \geq 1$. Systems with a strictly convect entropy).

For $\alpha>0$ in \eqref{RXENTPROP} there exists $\nu>0$ and symmetric, positive definite matrices $A_j$, $j=1,\dots, d$ such that
\begin{equation}\label{AHYPSCENTMD} \tag{H2*}
\begin{aligned}
\tfrac{1}{2} \;  \xi_j^{\top} \bigl( A_j\D^2{\eta}(u) + \D ^2{\eta}(u)A_i\bigr) \xi_j -\alpha \bigg|\sum_{i=1}^d \SP
\D F(u)^{\top} \xi_j \bigg|^2 \geq\nu \sum_{i=1}^d |\xi_j|\,, \\
\forall \xi_1,\dots, \xi_d \in \RR^n, \;\; u\in \RR^n.
\end{aligned}
\end{equation}
\end{itemize}

\subsection{Dissipation}

The following hypotheses will be relevant to our subsequent discussion.
\begin{itemize}
\item  The source term $G(u)$ is {\it weakly dissipative} in the sense of Definition \ref{D1.1}.
\begin{definition} \label{D1.1} \rm
We say that the source $G(u)$ is {\it weakly dissipative}, if
\begin{equation}\tag{H3-a}\label{WDSRC}
\begin{aligned}
-\big(\D\eta(u)-\D\eta(\bar{u})\big)\big(G(u)-G(\bar{u})\big)  \geq  0 \,, \quad u,\bar{u} \in \RR^n.\,
\end{aligned}
\end{equation}

\end{definition}
\end{itemize}
An alternative condition on the source $G$, exploited in Theorem \ref{gen-source},
reads:
\begin{itemize}

\item Suppose that for every compact set $\mathcal{A}$ there exists
$L_{\mathcal{A}}>0$ such that
\begin{equation}\tag{H3-b}\label{LISRC}
    |G(u)-G(\bar{u})| \leq L_{\mathcal{A}} |u-\bar{u}|\,, \quad u \in \RR^n,\, \bar{u}\in\mathcal{A}\,.
\end{equation}
\end{itemize}

\par\smallskip

Through out the article we use the assumptions that $G(0)=0$ and $G\in C(\RR^n)$.

\subsection{Source potential}

For the weakly dissipative source $G$ we employ an additional assumption that it is a gradient:
\begin{itemize}
\item  Suppose there exists a potential $R(u):\RR^n\to\RR$ such that
\begin{equation}\tag{H4}\label{POTN}
\begin{aligned}
&G(u)=-\D R(u)^{\top}\\[1pt]
&R(u) \geq R(0)=0, \; \D R(0)^{\top}=0 \\[2pt]
&|\D R(u)|=|G(u)| \, \leq \, C_R \bigl(1+R(u)\bigr), \quad u\in\RR^n.
\end{aligned}
\end{equation}
\end{itemize}

\section{Stability estimates}\label{S3}

\subsection{Systems with a strictly convex entropy $\eta$, $d=1$.}

The balance law \eqref{BLAWMDIM} for $d=1$ reads
\begin{equation}\label{BLAW1D}
\partial_{t}u + \partial_{x}F(u) =
G(u),\quad  u\in\RR^n
\end{equation}
and the relaxation model by
\begin{equation}\label{RXSYST1D}
\left\{
\begin{aligned}
\partial_{t}u+\partial_{x}v& \SP = \SP 0\\
\partial_{t}v + A\partial_{x}u& \SP = \SP -\f{1}{\eps}\Bigl(v - F(u) +\int^{x} G\bigl(u(x,t)\bigr)\,dx\Bigr)
\end{aligned}\right.
\end{equation}
with $A$ symmetric, positive definite matrix. In that case the second order relaxation system \eqref{RXBLAWMDIM} reads
\begin{equation}\label{RXBLAW1D}
\partial_{t}u+\partial_{x}F(u)=G(u)+\eps(Au_{xx}-u_{tt}).
\end{equation}
\par\smallskip

We now consider the hyperbolic system \eqref{BLAW1D} that is equipped with the entropy-entropy flux pair $\eta-q$, with $\eta$ strictly convex, and establish stability results for the relaxation model \eqref{RXBLAW1D}.

\begin{lemma} \label{energy}
Suppose $u \equiv u^{\eps}(x,t)$ is a smooth solution to the equation \eqref{RXBLAW1D} on $\RR \times [0,T]$, $\eta-q$ is
the entropy-entropy flux pair of \eqref{BLAW1D} and $\bar{\alpha}\in\RR$ is fixed. Then, the following energy identity holds
\begin{equation}\label{STABIDSCENT1D}
\begin{aligned}
&\partial_{t}\bigg[\eta(u+\eps
u_{t})+\tfrac{1}{2}\eps^2\bar{\alpha} \SP |u_{t}|^2+\eps^{2}\bar{\alpha}\SP
u_x ^{\top} A u_x\\
&\qquad+\eps^2u_{t}^{\top} \Big(\,\tfrac{1}{2}\bar{\alpha} {\bf I}-\!
\int_{0}^{1}\!\int_{0}^{s} \D^2{\eta}(u+\eps\tau u_{t})\,d\tau ds\Big)u_{t} \bigg]+\partial_{x} q(u)\\[3pt]
&\quad+\eps\bar{\alpha}\big|u_{t}+\D F(u)u_{x}\big|^2+\eps u_{t}^{\top} \Big(\bar{\alpha} I-\D^2{\eta}(u)\Big)u_{t}\\[3pt]
&\quad+\eps u_{x} ^{\top} \Big(\D^2{\eta}(u)A-\bar{\alpha} \D F(u)^{\top}\D F(u)\Big)u_{x}\\[4pt]
&= \, \partial_{x}\Big(\eps \D {\eta}(u) Au_{x}+2\eps^2 \bar{\alpha} \SP u_{t}^{\top}
Au_{x}\Big)+\D{\eta}(u) G(u)+2\SP\eps  \bar{\alpha} u_{t}^{\top}G(u).
\end{aligned}
\end{equation}
\end{lemma}

\begin{proof}
Multiplying \eqref{RXBLAW1D} by $u_{t}^{\top}$ we obtain
\begin{equation}\label{TID1SCENT}
\begin{aligned}
&\partial_{t}\Bigl(\tfrac{1}{2}\eps|u_{t}|^2+\tfrac{1}{2}\eps
u_{x}^{\top} Au_{x}\Bigr)+|u_{t}|^2+u_{t} ^{\top} \D F(u)u_{x}\\
 &  =\partial_{x}\Bigl(\eps u_{t} ^{\top}
Au_{x}\Bigr)+u_{t}^{\top} G(u).
\end{aligned}
\end{equation}
Similarly, multiplying \eqref{RXBLAW1D} by $\D{\eta}(u)$ we get
\begin{equation}\label{TID2SCENT}
\begin{aligned}
&\del_{t}\Big(\eta(u)+\eps \D{\eta}(u) u_{t}\Big)+\del_{x}\SP q(u) \\
& + \eps\Big(\big(\D^2{\eta}(u)u_{x}\big)^{\top}
Au_{x}-u_{t}^{\top} \D^2{\eta}(u)u_{t}\Big)\\
&=\partial_{x}\Bigl(\eps \D{\eta}(u) Au_{x}\Bigr)+\D{\eta}(u) G(u).
\end{aligned}
\end{equation}
Now, we multiply \eqref{TID1SCENT} by $2\bar{\alpha}\eps$, add $\eqref{TID2SCENT}$
and use the identity
\begin{equation}\label{CHEXPENT}
\eta(u+\eps u_{t})=\eta(u)+\eps \D{\eta}(u)u_{t}+\eps^2
u_{t}^{\top}\biggl(\int_{0}^{1}\!\!\int_{0}^{s}\!\D^2{\eta}(u+\eps\tau
u_{t})\,d\tau\,ds\biggr)u_{t}
\end{equation}
to deduce \eqref{STABIDSCENT1D}. This proves the lemma.
\end{proof}

{We now establish the stability of solutions $\{u^{\eps}\}$. For that we will require the matrix $A$ to satisfy the subcharacteristic condition \eqref{AHYPSCENT1D} and make use of hypotheses \eqref{WDSRC}-\eqref{POTN} to control the source $G$. In the sequel we will use the notation
\begin{equation}\label{PHIDEF}
    \varphi(t):=\int_{\RR}  |u^{\eps}|^2 + \eps^2|u^{\eps}_x|^2 + \eps^2|u^{\eps}_t|^2 \,dx\, 
\end{equation}
with $u^{\eps}$ denoting a solution of \eqref{RXBLAW1D}.}

\begin{proposition}[\bf Weakly dissipative source]\label{w-diss2}

Let $\{u^{\eps}(x,t)\}$ be a family of smooth solutions to the equation \eqref{RXBLAW1D} on $\RR \times [0,T]$. Suppose that $u \equiv u^{\eps}$ decays fast at infinity and that:

\begin{itemize}
\item[(a1)]  \eqref{RXENTPROP} holds true, and the positive definite, symmetric matrix $A$ is such that the subcharacteristic condition \eqref{AHYPSCENT1D} is valid.
\item[(a2)]

The conditions \eqref{WDSRC} and \eqref{POTN} for the source $G$ hold true.
 \end{itemize}

\par\smallskip

\noindent Then for all $t\in[0,T]$
\begin{equation}\label{PHIESTENTW1D}
  \varphi(t) + \eps \!\int_{\RR} R(u^{\eps}(x,t)) \,dx + \int_{0}^{t}\int_{\RR} |\D{\eta}(u)G(u)| \SP dx dt \, \leq \, C \Big( \varphi(0) + \eps \! \int_{\RR} R(u^{\eps}(x,0)) \,dx \Big)
\end{equation}
with $\varphi$ defined in \eqref{PHIDEF}
and $C=C(A,\alpha,\beta)>0$ independent of both $\eps$ and $T$.
\end{proposition}

\begin{proof} By \eqref{RXENTPROP} we have
\begin{equation}\label{ALPHADOM}
0 \SP \leq \SP u_{t}^{\top} \Bigl(\tfrac{1}{2}\alpha {\bf I}-\int_{0}^{1}\int_{0}^{s}\! \D^2{\eta}(u+\eps\tau
u_{t})\,d\tau ds\Bigr)u_{t} \SP \leq \SP \tfrac{1}{2}\alpha |u_t|^2.
\end{equation}
Then, integrating the identity \eqref{STABIDSCENT1D}, with $\bar{\alpha}=\alpha$, and using hypotheses \eqref{WDSRC}, \eqref{POTN} we obtain
\begin{equation}\label{STABIDSCENTWDSOURCE}
\begin{aligned}
&\int_{\RR}\Big(\eta(u+\eps u_{t})+\tfrac{1}{2}\eps^2\alpha|u_{t}|^2+\eps^2\alpha
u_{x}^{\top} Au_{x}+2 \eps \alpha R(u)\Big)\,dxdt\\
&\quad +\int_{0}^{t}\!\int_{\RR}\,\eps\nu|u_{x}|^2+\tfrac{1}{2}\eps \alpha|u_{t}|^2 \,dxdt\\
&\quad +\int_{0}^{t}\!\int_{\RR}\eps\alpha |u_{t}+\D F(u)u_{x}|^2+| \D{\eta}(u)G(u)| \, dxdt\\
&\quad\leq \int_{\RR} \Big(\eta(u_0+\eps
{u_0}_{t})+c\SP\eps^2\alpha|{u_0}_t|^2+\eps^2 {u_0}_x^{\top} A {u_0}_x+2\eps \alpha R(u_0)\Big)\, dx
\end{aligned}
\end{equation}
with $\frac{1}{2} \leq c \leq 1$. From \eqref{RXENTPROP} it follows that
\begin{equation}\label{PHIEQUIVENT}
c_{1} \SP \varphi(t) \, < \int_{\RR} \Big(\eta(u+\eps
u_{t})+\tfrac{1}{2}\alpha \eps^{2}|u_{t}|^2+\eps^2 \alpha u_{x}^{\top} A \SP u_{x} \Big)\SP dx <  c_2 \SP \varphi(t)
\end{equation}
for some $c_1,c_2$ that depend on $\alpha,\beta$ and $A$. Then, combining \eqref{STABIDSCENTWDSOURCE} and \eqref{PHIEQUIVENT}, we obtain \eqref{PHIESTENTW1D}.

\end{proof}

\begin{proposition}[\bf General source]\label{Lgen-source-2}
Let $\{u^{\eps}(x,t)\}$ be a family of smooth solutions to the equation \eqref{RXBLAW1D} on $\RR \times [0,T]$. Suppose that  $u \equiv u^{\eps}$ decays fast at infinity and that:

\begin{itemize}
\item[(a1)]  \eqref{RXENTPROP} holds true, and the positive definite, symmetric matrix $A$, is such that the subcharacteristic condition \eqref{AHYPSCENT1D} is valid.
\item[(a2)] The condition \eqref{LISRC} for the source $G$ holds true.
 \end{itemize}

\par\smallskip

\noindent Then,
\begin{equation}\label{PHIESTSCENTG1D}\vspace{4pt}
  \varphi(t)  \SP \leq \SP  C\SP \varphi(0) \,, \quad t\in[0,T]
\end{equation}
with $\varphi$ defined in \eqref{PHIDEF}
and $C=C(A,\alpha,\beta,T,L)>0$ independent of $\eps$.
\end{proposition}

\begin{proof}
Integrating the energy identity \eqref{STABIDSCENT1D}, with $\bar{\alpha}=\alpha$, and using \eqref{RXENTPROP}, \eqref{AHYPSCENT1D} together with relations \eqref{ALPHADOM}, \eqref{PHIEQUIVENT} we obtain for $\tau \in [0,T]$
\begin{equation}\label{TSTESTSCENTGS1}
\begin{aligned}
c_1 \varphi(\tau) \leq c_2 \SP \varphi(0) + \int^{\tau}_{0}\!\int_{\RR}\Big(\D \eta (u)G(u)+2 \eps \alpha u_t^{\top}G(u)\Bigr) dx  dt\,.
\end{aligned}
\end{equation}
Since $G(0)=0$, \eqref{LISRC} implies $|G(u)|\leq L|u|$ and therefore, in view of \eqref{RXENTPROP},
\begin{equation*}
|D\eta(u)G(u)+2\eps \alpha u_t^{\top}G(u)|\leq\alpha L|u|^2 +     \alpha\eps^2|u_{t}|^2+\alpha L^2|u|^2\,.
\end{equation*}
Then \eqref{PHIDEF}
and \eqref{TSTESTSCENTGS1} imply
\begin{equation*}
\begin{aligned}
 \varphi(\tau) \leq c \Big( \SP \varphi(0) + \int^{\tau}_{0} \varphi(t) \SP dt \Big)
\end{aligned}
\end{equation*}
with $c>0$ depending on $c_1,c_2$ and $L$. Then, we conclude \eqref{PHIESTSCENTG1D} via the Gronwall
lemma.
\end{proof}

\section{Compactness properties}\label{S4}
This section focuses on systems of hyperbolic balance laws with a strictly convex entropy, $d=1$.
\subsection{ Systems with a strictly convex entropy, $d=1$.}
Starting with a family  $\{u_\eps \}$  smooth solutions of \eqref{RXBLAW1D} on $\RR  \times [0, \infty)$ the goal in this section  is to control the dissipation measure and to show that
\begin{equation*}
\del_t \bar{\eta}(u^{\eps}) +  \del_x \bar{q}(u^{\eps})\,\, \mbox{lies in a compact set of }\,\, H^{-1}_{loc}(\RR \times \RR^+)
\end{equation*}
for a certain class of entropy-entropy flux pairs $(\bar{\eta}, \bar{q})$.
In the proof we use Murat's lemma \cite{Murat-1981}.
\begin{lemma}[\bf Murat's Lemma \cite{Murat-1981}]\label{FM}
 Let ${\mathcal O}$ be an open subset of $\RR^m$ and $\{\phi_j\}$ a bounded
sequence of $W^{-1,p} ({\mathcal O})$  for some $p > 2.$ In addition let ${\phi}_j = \chi_j  + \psi_j$, where $\{\chi_j\}$
belongs in a compact set of $H^{-1}({\mathcal O})$ and $\{\psi_j\}$ belongs in a bounded set of the space of measures $M({\mathcal O})$. Then $\{\phi_j\}$ belongs in a compact set of $H^{-1}({\mathcal O}).$
\end{lemma}
In the presence of uniform $L^{\infty}$-bounds the compensated compactness framework (cf. Tartar \cite{Tartar-1979}, DiPerna \cite{DiPerna83}) guarantees compactness of approximate solutions and implies that, along a subsequence, $u_{\eps} \to u\,\, a.e.\, (x,t).$  In the absence of  $L^{\infty}$-bounds, convergence
of viscosity approximations in the literature has been established in the context of elastodynamics by Shearer \cite{Sh} and Shearer and Serre \cite{SSh}. The objective in that context is to establish the reduction of the generalized Young measure to a point mass and to show strong convergence.

\begin{theorem}[{\bf Weakly dissipative source}] \label{COMP-WD-1D}
Let $\{u_\eps \}$ be a family of smooth solutions of \eqref{RXBLAW1D}  on $\RR  \times [0, T]$ emanating from smooth initial data. The family $\{u_\eps\}$ is assumed to decay fast at infinity. Let the hypotheses of Proposition \ref{w-diss2} remain valid so that the stability estimate \eqref{PHIESTENTW1D}  holds true with  $\eta-q$ entropy-entropy flux pair satisfying \eqref{RXENTPROP}, and $A$ a symmetric, positive-definite matrix subject to \eqref{AHYPSCENT1D}.
Then, for entropy pairs $(\bar{\eta}, \bar{q})$ satisfying
\begin{equation}\|\bar{\eta}\|_{L^{\infty}}, \|\bar{q}\|_{L^{\infty}}, \|\D\bar{\eta}\|_{L^{\infty}}, \|\D^2\bar{\eta}\|_{L^{\infty}} \le C \label{entropy-hyp}
\end{equation}
and
\begin{equation}\label{proj-growth}
  |\D \bar{\eta}(v) G(v)| \leq C\big(M-\D \eta(v) G(v)\big), \quad \forall v\in\RR^n\,,
\end{equation}
the family
\begin{equation}\label{measure-cc}
\big\{\del_t \bar{\eta}(u^{\eps}) +  \del_x \bar{q}(u^{\eps})\big\}_{\eps}\,\, \mbox{lies in a compact set of }\,\, H^{-1}_{loc}(\RR \times [0,T]).
\end{equation}
\end{theorem}
\begin{proof}
Let $\{u^\eps \}$ be a family of smooth solutions of \eqref{RXBLAW1D} on $\RR  \times [0, T].$
The goal is to control the dissipation measure  and to establish  \eqref{measure-cc} for a class of entropy-entropy flux pairs $(\eta, q)$. It suffices to establish \eqref{measure-cc} for entropy-entropy flux pairs $(\bar{\eta},\bar{q})$ satisfying \eqref{entropy-hyp}.
This class of entropy pairs has been  used in the literature in a different context in order to establish the reduction of the generalized Young measure to a point mass and to show strong convergence.

Starting from \eqref{RXBLAW1D} we obtain
\begin{equation}
\begin{aligned}
&\del_t \bar{\eta}(u^{\eps}) + \del_x \bar{q}(u^{\eps}) = \eps \del_x \left( \D\bar{\eta}(u^{\eps}) A u^{\eps}_x \right) - \eps \del_t \left( \D\bar{\eta}(u^{\eps}) u^{\eps}_t\right) \\[3pt]
&\qquad\qquad\quad\qquad\qquad-\eps {u^{\eps}_x} ^{\top} \D^2\bar{\eta}(u^{\eps}) A u^{\eps}_x + \eps {u^{\eps}_t} ^{\top} \D^2 \bar{\eta}(u^{\eps}) u^{\eps}_t + \D\bar{\eta}(u^{\eps}) G(u^{\eps})\\[3pt]
&\qquad\qquad\qquad:= I_1 + I_2+ I_3 + I_4 + I_5. \nonumber
\end{aligned}
\end{equation}
From \eqref{PHIESTENTW1D} and \eqref{entropy-hyp}, the terms $I_1, I_2$ lie in  compact set of $H^{-1}$, and the terms $I_3, I_4$ are bounded in $L^1$. By \eqref{PHIESTENTW1D} $\D \eta G(u^{\eps})$ is bounded in $L_1$ by initial data (due to the weak dissipation assumption) and therefore by \eqref{proj-growth} the term $\D \bar{\eta}G(u^{\eps})$ is in a bounded set of $L_1$ as well. Therefore, by Murat's lemma \cite{Murat-1981}, the sum $\sum I_i$ lies in a bounded set of $W^{-1, \infty}.$
\end{proof}

Following, similar line of argument an analogous result for systems of hyperbolic balance laws with general source term is established.

\begin{theorem}[{\bf General source}] \label{COMP-LIPS-1D}
Let $\{u_\eps \}$ be a family of smooth solutions of \eqref{RXBLAW1D} on $\RR  \times [0, T]$ emanating from smooth initial data. The family $\{u_\eps\}$ is assumed to decay fast at infinity. Let the hypothesis of Proposition  \ref{Lgen-source-2} remain valid so that the stability estimate \eqref{PHIESTSCENTG1D} holds true with $\eta-q$ entropy-entropy flux pair satisfying \eqref{RXENTPROP}, and $A$ a symmetric, positive-definite matrix subject to \eqref{AHYPSCENT1D}.
Then, for entropy pairs $(\bar{\eta}, \bar{q})$ satisfying
\begin{equation}\|\bar{\eta}\|_{L^{\infty}}, \|\bar{q}\|_{L^{\infty}}, \|\D\bar{\eta}\|_{L^{\infty}}, \|\D^2\bar{\eta}\|_{L^{\infty}} \le C \label{entropy-hyp2}
\end{equation}
the family
\begin{equation}\label{measure-cc2}
\Big\{\del_t \bar{\eta}(u_{\eps}) +  \del_x \bar{q}(u_{\eps})\Big\}_{\eps}\,\, \mbox{lies in a compact set of }\,\, H^{-1}_{loc}(\RR \times [0,T]).
\end{equation}
\end{theorem}
\begin{proof}
The proof follows similar line of argument as the one in Theorem \ref{COMP-WD-1D}.
Starting from \eqref{RXBLAW1D} we obtain
\begin{equation}
\begin{aligned}
\del_t \bar{\eta}(u^{\eps}) + \del_x \bar{q}(u^{\eps}) &= \eps \del_x \left( \D\bar{\eta}(u^{\eps}) A u^{\eps}_x \right) - \eps \del_t \left( \D\bar{\eta}(u^{\eps}) u^{\eps}_t\right) \\[3pt]
&\qquad-\eps {u^{\eps}_x} ^{\top} \D^2\bar{\eta}(u^{\eps}) A u^{\eps}_x + \eps {u^{\eps}_t} ^{\top} \D^2 \bar{\eta}(u^{\eps}) u^{\eps}_t + \D\bar{\eta}(u^{\eps}) G(u^{\eps})\\[3pt]
&:= I_1 + I_2+ I_3 + I_4 + I_5. \nonumber
\end{aligned}
\end{equation}
From \eqref{PHIESTSCENTG1D} and \eqref{entropy-hyp2}, the terms $I_1, I_2$ lie in  compact set of $H^{-1},$ the terms $I_3, I_4$ are bounded in $L^1$, whereas
$$
|I_5| = |\D\bar{\eta}(u) ^{\top} G(u)| \leq c |u|^2
$$
is by Lemma \ref{Lgen-source-2} bounded in $L^1.$
Therefore, by Murat's lemma \cite{Murat-1981}, the sum $\sum I_i$ lies in a bounded set of $W^{-1, \infty}.$
\end{proof}

\section{Error estimates via the relative entropy method, $d=1$}\label{S5}

{ \black In this section, we establish the convergence of solutions of \eqref{RXBLAW1D} to solutions of \eqref{BLAW1D}  before the formation of shocks. In the spirit of \cite{AMT-2005} we use the modified relative entropy method \cite{Dafermos10} by introducing a functional $H^{rel}(\bar{u},u^{\eps}),$ which  monitors the difference between the solutions $\bar{u}$ to the equilibrium and the solutions $u^{\eps}$ to the relaxation systems. The presence of the source $G$ in our work, however, requires us to modify the method significantly. More specifically, in order to treat the  weakly dissipative source $G$, which  satifies  typically no growth restrictions, we need to introduce the relative potential $R^{rel}(\bar{u},u^{\eps})$ (see \eqref{RREL} in Section 6.3). This potential becomes  part of a Lyapunov functional monitoring the evolution of the difference between the two sources. In the case of the general source $G$ the terms associated with the source  are treated as error (cf.\ Section 6.4).}

\subsection{The decay functional and relative entropy identity.}

\label{S8.1}

Let $\eta-q$ be an entropy-entropy flux pair satisfying \eqref{RXENTPROP}. We define the corresponding relative entropy-entropy flux pair by
\begin{equation}\label{RENTPAIR}\vspace{3pt}
\begin{aligned}
H^{rel}(\bar{u},u^{\eps}) &= \eta \bigl( u^{\eps}+\eps(u^{\eps}-\bar{u})_t \bigr) - \eta(\bar{u}) - \D\eta(\bar{u})
\bigl( u^{\eps}+\eps(u^{\eps}-\bar{u})_t -\bar{u}
\bigr)\\[3pt]
Q^{rel}(\bar{u},u^{\eps}) &= q(u^{\eps}) - q(\bar{u}) - \D\eta(\bar{u})(F(u^{\eps})-F(\bar{u}))
\end{aligned}
\end{equation}
and set the functional
\begin{equation}\label{GFUNCTIONALDEF}
\begin{aligned}
\mathcal{G}(\bar{u},u^{\eps}) & = H_R(\bar{u},u^{\eps}) \\
&\quad + \eps^2  (u^{\eps}-\bar{u})_t ^{\top}\bigl(\alpha I - \overline{D^2 \eta}\bigr) (u^{\eps}-\bar{u})_t \\
 &\quad +\eps^2 \alpha \SP (u^{\eps}-\bar{u})_x^{\top}A(u^{\eps}-\bar{u})_x \,,
\end{aligned}
\end{equation}
where $A$ is a symmetric, positive definite matrix, $\alpha >0$ a fixed constant defined in \eqref{RXENTPROP} and
\begin{equation}\label{AVGD2ETA}
\overline{D^2 \eta} = \int_0^1 \int_0^s \, D^2 \eta \bigl(u^{\eps}+\eps \tau
(u^{\eps}-\bar{u})_t \bigr) \, d\tau ds\,.
\end{equation}

\par\smallskip

\begin{lemma}[{\bf Relative entropy identity}]\label{RENTIDENTITY}
Let $\bar{u}$, $u^{\eps}$ be smooth solutions to \eqref{BLAW1D}, \eqref{RXBLAW1D},
respectively. Then, the following energy
identity holds
\begin{equation}\label{LFI}
\begin{aligned}
&\del_t \DSP \mathcal{G} (\bar{u},u^{\eps}) + \del_x \SP Q^{rel} (\bar{u},u^{\eps}) + \eps \alpha \bigl|(u^{\eps}-\bar{u})_t+\D F(u^{\eps})(u^{\eps}-\bar{u})_x\bigr|^2\\
&\quad + \eps \Bigl\{ (u^{\eps}-\bar{u})_t^{\top}\Big(\alpha I - \D ^2\eta(u^{\eps})\Big) (u^{\eps}-\bar{u})_t   \Bigr\}\\
&\quad + \eps\Big\{(u^{\eps}-\bar{u})_x^{\top} \Big(\D ^2\eta(u^{\eps}) A - \alpha \SP \D F(u^{\eps})^{\top}\D F(u^{\eps}) \Big)(u^{\eps}-\bar{u})_x \Bigr\}\\
& = \del_x \Bigl \{ \eps \big(\D \eta(u^{\eps}) - \D \eta(\bar{u})\big) A(u^{\eps}-\bar{u})_x + 2\alpha \eps^2 \big(A(u^{\eps}-\bar{u})_x\big)^{\top} (u^{\eps}-\bar{u})_t \Bigr\}\\
&\quad - \big(\D ^2 \eta(\bar{u}) \bar{u}_x\big)^{\top} \Bigl( F(u^{\eps})-F(\bar{u})-\D F(\bar{u})(u^{\eps}-\bar{u}) \Bigr) \\
&\quad + \big(a_{1}+a_{2}+b_{1}+b_{2}+ 2 \SP \eps \alpha
(c_{1}+c_{2})\big)+ \big(d_1 + d_2 + 2 \SP \eps  \alpha  d_3 \big)\, ,
\end{aligned}
\end{equation}
where
\begin{equation}\label{ERRT1}\vspace{3pt}
\begin{aligned}
a_{1} &= \eps \big(\bigl( \D^2 \eta(u^{\eps})- \D^2 \eta(\bar{u}) \bigr) \bar{u}_t \big)^{\top} (u^{\eps}-\bar{u})_t \\
a_{2} &= -\eps \bigl( \D\eta(u^{\eps}) - \D\eta(\bar{u})\bigr) \bar{u}_{tt}\\
b_{1} &= \eps \big(\big( \D^2 \eta(u^{\eps}) - \D^2 \eta(\bar{u}) \bigr)\bar{u}_x\big)^{\top} A(u^{\eps}-\bar{u})_x \\
b_{2} &= -\eps \bigl( \D\eta(u^{\eps}) - \D\eta(\bar{u}) \bigr) A \bar{u}_{xx}\\
c_{1} &= \eps (A\bar{u}_{xx}- \bar{u}_{tt})^{\top} (u^{\eps}-\bar{u})_t\\
c_{2} &= - \big(\bigl(\D F(u^{\eps})-\D F(\bar{u}) \bigr)\bar{u}_x\big)^{\top} (u^{\eps}-\bar{u})_t
\end{aligned}
\end{equation}
are the error terms, and
\begin{equation}\label{ERRT2}
\begin{aligned}
d_1 &= \bigl( \D\eta(u^{\eps}) - \D\eta(\bar{u}) \bigr)(G(u^{\eps})-G(\bar{u}))\\
d_2 &= G(\bar{u})^{\top} \Big( \D\eta(u^{\eps})-\D\eta(\bar{u})-\D^2\eta(\bar{u})(u^{\eps}-\bar{u}) \Big) \\
d_3 &= \bigl( G(u^{\eps})-G(\bar{u}) \bigr)^{\top} (u^{\eps}-\bar{u})_t\,.
\end{aligned}
\end{equation}
are the terms associated with the source $G$.
\end{lemma}

\begin{proof}
By \eqref{RXENTPROP}, \eqref{BLAW1D}, and \eqref{RXBLAW1D} we have
\begin{equation}\vspace{3pt}\label{LID1}
\begin{aligned}
&\del_t \bigl( \eta(u^{\eps}) - \eta(\bar{u}) \bigr) + \del_x \bigl( q(u^{\eps}) - q(\bar{u})
\bigr) \\[3pt]
&\quad = \eps \bigl( D\eta(u^{\eps})Au^{\eps}_{xx} - D\eta(u^{\eps})u^{\eps}_{tt} \bigr) + D\eta(u^{\eps}) G(u^{\eps}) -
D\eta(\bar{u}) G(\bar{u})\,.
\end{aligned}
\end{equation}
Similarly,
$$
\del_t (u^{\eps}-\bar{u}) + \del_x (F(u^{\eps})-F(\bar{u})) = \eps (Au^{\eps}_{xx}-u^{\eps}_{tt})+ G(u^{\eps})-G(\bar{u})
$$
and hence, after multiplying the above identity by $\D\eta(\bar{u})$, we have
\begin{equation}\label{LID2}
\begin{aligned}
\del_t & \big( \D\eta(\bar{u}) (u^{\eps}-\bar{u}) \big) + \del_x \big( \D\eta(\bar{u}) \bigl(F(u^{\eps})-F(\bar{u}) \bigr)\big)\\
&= \big(\D^2\eta(\bar{u})\bar{u}_t\big)^{\top} (u^{\eps}-\bar{u}) + \del_x \big( \D \eta(\bar{u}) \big) (F(u^{\eps})-F(\bar{u})) \\[3pt]
&  \quad + \eps \big(\D\eta(\bar{u})Au^{\eps}_{xx}-\D\eta(\bar{u})u^{\eps}_{tt} \big) + \D\eta(\bar{u}) (G(u^{\eps})-G(\bar{u}))\,.
\end{aligned}
\end{equation}

\par\smallskip

The existence of the entropy-pair $\eta-q$ is equivalent to the property
\begin{equation*}
 \D^2 \eta(v) \D F(v) = \D F(v)^{\top} \D^2\eta(v), \quad \forall v\in\RR^n
\end{equation*}
and hence, using \eqref{BLAW1D},  the first term on the right-hand side of \eqref{LID2} can be expressed as
\begin{equation*}
\begin{aligned}
\big(\D^2 \eta(\bar{u})\bar{u}_t\big)^{\top} (u^{\eps}-\bar{u})
& = -\big(\D^2 \eta(\bar{u})\bar{u}_x\big)^{\top} \D F(\bar{u})(u^{\eps}-\bar{u}) + G(\bar{u})^{\top} \D^2\eta(\bar{u})(u^{\eps}-\bar{u})\,.
\end{aligned}
\end{equation*}
Combining \eqref{LID1}, \eqref{LID2} and the above identity we obtain
\begin{equation}\label{RENTID1} \vspace{3pt}
\begin{aligned}
&\!\!\!\!\!\!\!\!\!\del_t \big( \eta(u^{\eps}) - \eta(\bar{u}) - D\eta(\bar{u}) (u^{\eps}-\bar{u}) \big) \\
\qquad &\quad +\del_x \big( q(u^{\eps}) - q(\bar{u}) - \D\eta(\bar{u}) (F(u^{\eps}) - F(\bar{u}))
\big)\\
& \qquad\qquad = - \big(\D^2 \eta(\bar{u})\bar{u}_x \big)^{\top} \bigl( F(u^{\eps})- F(\bar{u}) - \D F(\bar{u})(u^{\eps}-\bar{u}) \big)\\
& \SP\qquad\quad\qquad + \eps \bigl( \D\eta(u^{\eps}) - \D\eta(\bar{u})\bigr) Au^{\eps}_{xx} - \eps \bigl( \D\eta(u^{\eps}) - \D\eta(\bar{u})\bigr) u^{\eps}_{tt}\\
& \SP\qquad\quad \qquad + \bigl( \D\eta(u^{\eps}) - \D\eta(\bar{u})\bigr) \bigl( G(u^{\eps}) - G(\bar{u}) \bigr) \\
& \SP\qquad \quad\qquad+  G(\bar{u})^{\top} \bigl( \D\eta(u^{\eps}) - \D\eta(\bar{u}) - \D^2\eta(\bar{u}) (u^{\eps}-\bar{u})\bigr).\\
\end{aligned}
\end{equation}
Next, we express the second and third terms on the right-hand side of \eqref{RENTID1} as
\begin{equation*}
\begin{aligned}
\bigl(  \D\eta(u^{\eps}) & -  \D\eta(\bar{u})\bigr)  u^{\eps}_{tt} & \\
& \, = \, \del_t \Big(\bigl(\D\eta(u^{\eps}) - \D\eta(\bar{u})\bigr) (u^{\eps}-\bar{u})_t \Big) - \big( \D^2 \eta(u^{\eps})(u^{\eps}-\bar{u})_t \big)^{\top} (u^{\eps}-\bar{u})_t\\
& \, \quad - \big(\bigl( \D^2 \eta(u^{\eps}) - \D^2 \eta(\bar{u}) \bigr)\bar{u}_t \big)^{\top}
(u^{\eps}-\bar{u})_t+ \bigl( \D\eta(u^{\eps}) - \D\eta(\bar{u}) \big) \bar{u}_{tt}\\
\end{aligned}
\end{equation*}
\begin{equation*}
\begin{aligned}
\bigl( \D\eta(u^{\eps}) & -  \D\eta(\bar{u})\bigr) Au^{\eps}_{xx}\\
& \, = \, \del_x \Bigl(\bigl( \D\eta(u^{\eps}) - \D\eta(\bar{u})\bigr)
A(u^{\eps}-\bar{u})_x \Bigr) - \big( \D^2 \eta(u^{\eps}) (u^{\eps}-\bar{u})_x \big)^{\top} A(u^{\eps}-\bar{u})_x\\
& \, \quad - \big(\bigl( \D^2\eta(u^{\eps}) - \D^2 \eta (\bar{u})\bigr) \bar{u}_x \big)^{\top} A(u^{\eps}-\bar{u})_x + \bigl( \D\eta(u^{\eps})
- \D\eta(\bar{u})\bigr)^{\top} A \bar{u}_{xx}
\end{aligned}
\end{equation*}
and observe that
\begin{equation*}
\begin{aligned}
\eta\bigl(u^{\eps} + \eps (u^{\eps}-\bar{u})_t\bigr) = \eta(u^{\eps}) + \eps \D \eta(u^{\eps})  (u^{\eps}-\bar{u})_t + \eps^2
(u^{\eps}-\bar{u})_t^{\top} \overline{\D ^2 \eta} \SP (u^{\eps}-\bar{u})_t\,.
\end{aligned}
\end{equation*}
Then \eqref{RENTPAIR}, \eqref{RENTID1}  and the last three identities imply \begin{equation}\label{RENTID2}
\begin{aligned}
&\!\!\!\del_t \Bigl\{ H^{rel}(\bar{u},u^{\eps}) - \eps^2 (u^{\eps}-\bar{u})_t^{\top} \SP \overline{\D ^2
\eta} \SP (u^{\eps}-\bar{u})_t \Bigr\} + \del_x  Q^{rel}(\bar{u},u^{\eps}) \\
& \quad + \eps \Bigl\{ \big(\D ^2\eta(u^{\eps})(u^{\eps}-\bar{u})_x \big)^{\top}  (u^{\eps}-\bar{u})_x -
\big(\D ^2\eta(u^{\eps})(u^{\eps}-\bar{u})_t\big)^{\top} (u^{\eps}-\bar{u})_t \Bigr\}\\
& \, = \, \del_x \Bigl\{ \eps (\D \eta(u^{\eps}) - \D \eta(\bar{u}))
A(u^{\eps}-\bar{u})_x\Bigr\}\\
& \,\qquad - \big(\D ^2\eta(\bar{u})\bar{u}_x \big)^{\top} \Big( F(u^{\eps})-F(\bar{u}) - \D F(\bar{u})(u^{\eps}-\bar{u}) \Big)\\
& \,\qquad + a_{1t}+a_{2t}+b_{1x}+b_{2x}+d_1 +d_2
\end{aligned}
\end{equation}
with the last six terms on the right-hand side of the above identity defined in
\eqref{ERRT1}$_{1,2,3,4}$ and \eqref{ERRT2}$_{1,2}$.

\par\smallskip

The identity \eqref{RENTID2} is supplemented by a correction accounting for the fact that
the third term is indefinite. The correcting identity is obtained by multiplying the
equation
\begin{equation*}
\begin{aligned}
    (u^{\eps}-\bar{u})_t +  \D F(u^{\eps})(u^{\eps}-\bar{u})_x & =\eps A (u^{\eps}-\bar{u})_{xx} - \eps(u^{\eps}-\bar{u})_{tt}+ \eps (A \bar{u}_{xx} - \bar{u}_{tt})      \\
    & \quad + (G(u^{\eps})-G(\bar{u})) - (\D F(u^{\eps})-\D F(\bar{u}))\bar{u}_x
\end{aligned}
\end{equation*}
by $(u^{\eps}-\bar{u})_t$ and integrating by parts, which leads to
\begin{equation}\label{CORRECTTERM}
\begin{aligned}
 \del_t \Bigl\{   & \tfrac{1}{2}\,  \eps |u^{\eps}_t-\bar{u}_t|^2   + \tfrac{1}{2}
\,\eps  (u^{\eps}-\bar{u})_x^{\top} \SP A(u^{\eps}-\bar{u})_x\Bigr\}\\
& + |u^{\eps}_t-\bar{u}_t|^2 + \big(\D F(u^{\eps})(u^{\eps}-\bar{u})_x \big)^{\top} (u^{\eps}-\bar{u})_t\\
&\quad \qquad\qquad \, = \, \del_x \Bigl\{ \eps \big(A (u^{\eps}-\bar{u})_x\big)^{\top} (u^{\eps}-\bar{u})_t\Bigr\} + c_{1t} +
c_{2t} + d_3
\end{aligned}
\end{equation}
with the terms on the right-hand side defined in \eqref{ERRT1}$_{5,6}$ and
\eqref{ERRT2}$_3$.

\par\smallskip

Finally, multiplying \eqref{CORRECTTERM} by $2\alpha\eps$ and adding the resulting identity to
\eqref{RENTID2} we obtain \eqref{LFI}.
\end{proof}

\subsection{Preliminary estimate of $\mathcal{G}$}

In this section we establish a preliminary estimate for the functional $\mathcal{G}(\bar{u},u^{\eps})$ employed in the proofs of Theorem \ref{WDESTTHM} and Theorem \ref{gen-source}. For this purpose we define
\begin{equation} \label{PSIDEF}
    \Psi(t):=\int_{\RR} \, |\bar{u}-u^{\eps}|^2 + \eps^2|(\bar{u}-u^{\eps})_x|^2 + \eps^2|(\bar{u}-u^{\eps})_t|^2\,dx\,
\end{equation}
used in our further analysis, where $\bar{u}$, $u$ are smooth solutions to the equilibrium and relaxation system, respectively.

\begin{lemma}\label{GESTLMM}
Let $\bar{u}$, $u^{\eps}$ be smooth solutions of \eqref{BLAW1D}, \eqref{RXSYST1D}, respectively and
suppose that both $\bar{u}$, $u^{\eps}$ decay sufficiently fast at infinity. Suppose
that:

\begin{itemize}
\item[(a1)]  \eqref{RXENTPROP}  holds true and  the positive definite, symmetric matrix $A$, is such that the subcharacteristic condition \eqref{AHYPSCENT1D} is valid.
\item[$(a2)$] For some $M>0$
\begin{equation*}
|D^2F(u)| \leq M, \quad |D^3\eta(u)| \leq M, \quad u\in \RR^n.
\end{equation*}
 \end{itemize}
Then,
\begin{itemize}
\item[$(i)$] There exists $c_1,c_2>0$ independent of $\eps>0$ such that
\begin{equation}\label{TP1a}
    c_1 \Psi(t) \, \leq \, \int_{\RR} \,\mathcal{G}(\bar{u},u^{\eps})\, dx
    \, \leq c_2 \Psi(t)\,.
\end{equation}

\item[$(ii)$] The functional $\mathcal{G}(\bar{u},u^{\eps})$ is positive definite and
satisfies
\begin{equation}\label{GFUNCEVOLPiii}
\begin{aligned}
     \frac{d}{dt} & \int_{\RR} \,\mathcal{G}(\bar{u},u^{\eps})\,dx + \eps \bar{c}
    \int_{\RR} \, |u^{\eps}_x-\bar{u}_x|^2 + |u^{\eps}_t-\bar{u}_t|^2 \,dx\\[2pt]
    & \leq \, C(T,\bar{u}) \Bigl( \Psi(t) + \eps^2\Bigr) + \int_{\RR}
    \Bigl(
    d_1+ d_2 + 2 \eps \alpha d_3 \Bigr) dx
\end{aligned}
\end{equation}
with $d_1,d_2,d_3$ defined in
\eqref{ERRT2} and $C=C(T,\bar{u})>0$ independent of $\eps$.
\end{itemize}

\end{lemma}
\begin{proof}
From \eqref{RXENTPROP} and 
\eqref{RENTPAIR}$_1$ we have
\begin{equation}\label{RENTEQUIV}
    \beta |u^{\eps}+\eps(u^{\eps}-\bar{u})_t-\bar{u}|^2 \SP \leq \SP H^{rel}(\bar{u},u^{\eps})\, \leq \, \alpha |u^{\eps} +
    \eps(u^{\eps}-\bar{u})_t-\bar{u}|^2\,.
\end{equation}
Also, \eqref{RXENTPROP}  and \eqref{AVGD2ETA} imply
\begin{equation}\label{ALPHADOM2}
\alpha {\bf I} \SP \geq \SP \alpha \SP {\bf I} - \overline{D^2\eta}  = \alpha I - \int_0^1 \int_0^s \, D^2 \eta
\bigl(u^{\eps}+\eps \tau (u^{\eps}-\bar{u})_t \bigr) \, d\tau ds \, \geq \,
\tfrac{1}{2}\alpha I.
\end{equation}
Combining \eqref{GFUNCTIONALDEF}, \eqref{RENTEQUIV}, \eqref{ALPHADOM2}  and recalling that $A$ is symmetric, positive definite we get \eqref{TP1a}.

\par\smallskip

Next, integrating \eqref{LFI}  we use \eqref{AHYPSCENT1D} and \eqref{ALPHADOM2} to conclude
\begin{equation}\label{TP2}
\begin{aligned}
\frac{d}{dt}  & \int_{\RR}  \,\mathcal{G}(\bar{u},u^{\eps})\, dx + \eps \bar{c}
\int_{\RR} \, |(\bar{u}-u^{\eps})_x|^2 + |(\bar{u}-u^{\eps})_t|^2 \,
dx\\
& \leq \int_{\RR} \Bigl\{ \bigl|\big(\D^2\eta(\bar{u})\bar{u}_x\big)^{\top}
\bigl(F(u^{\eps})-F(\bar{u})-\D F(\bar{u})(u^{\eps}-\bar{u})\bigr)\bigr|\\
&\qquad+\bigl|a_{1}+a_{2}+b_{1}+b_{2}+2\eps\alpha(c_{1}+c_{2})\bigr|
+d_1+d_2+2\SP\eps\alpha d_3 \Bigr\} \, dx
\end{aligned}
\end{equation}
By $(a3)$ we have
\begin{equation*}
\begin{aligned}
\int_{\RR} \,  \bigl| & \big(\D^2\eta(\bar{u})\bar{u}_x \big)^{\top} \bigl(
F(u^{\eps})-F(\bar{u})-\D F(\bar{u})(u^{\eps}-\bar{u})\bigr)\bigr| \,dx \,
\leq \, C \BBN{u^{\eps}-\bar{u}}_{L^2}^2
\end{aligned}
\end{equation*}
and the error terms in \eqref{ERRT1} are estimated by
\begin{equation*}
\begin{aligned}
\BBN{a_{1}}_{L^1} &\leq \eps C \BBN{\bar{u}_t}_{L^{\infty}}
\BBN{u^{\eps}-\bar{u}}_{L^2} \BBN{u^{\eps}_t-\bar{u}_t}_{L^2}\\
\BBN{b_{1}}_{L^1} & \leq \eps C \BBN{\bar{u}_{x}}_{L^{\infty}}\BBN{u^{\eps}-\bar{u}}_{L^2} \BBN{u^{\eps}_x
-\bar{u}_x}_{L^2}
\end{aligned}
\quad
\begin{aligned}
\BBN{a_{2}}_{L^1}  &\leq \eps C \BBN{\bar{u}_{tt}}_{L^2} \BBN{u^{\eps}-\bar{u}}_{L^2}\\
\BBN{b_{2}}_{L^1}  &\leq \eps C \BBN{\bar{u}_{xx}}_{L^2} \BBN{u^{\eps}-\bar{u}}_{L^2}
\end{aligned}
\end{equation*}
and \vspace{3pt}
\begin{equation*}\vspace{3pt}
\begin{aligned}
\BBN{\eps c_{1}}_{L^1} &\leq \eps^2 C \bigl(\BBN{\bar{u}_{tt}}_{L^2} + \BBN{\bar{u}_{xx}}_{L^2}
\bigr)
\BBN{u^{\eps}_t-\bar{u}_t}_{L^2}\\
\BBN{\eps c_{2}}_{L^1} & \leq \eps \SP C \BBN{\bar{u}_{x}}_{L^{\infty}} \BBN{u^{\eps}-\bar{u}}_{L^2}
\BBN{u^{\eps}_t-\bar{u}_t}_{L^2}\,,
\end{aligned}
\end{equation*}
where $C$ is a generic constant that depends on $\alpha$, $M$ and norms of $\bar{u}$. Then,  by
\eqref{TP1a}, \eqref{TP2} and the above estimates we obtain \eqref{GFUNCEVOLPiii}.
\end{proof}


\subsection{Error estimates. Weakly dissipative source $G(u)$}

{\black To establish the convergence result for weakly dissipative source we introduce the {\it relative potential}
\begin{equation}\label{RREL}
   R^{rel}(\bar{u},u^{\eps}):= R(u^{\eps}) - R(\bar{u})  - \D R(\bar{u})(u-\bar{u}) \geq 0
\end{equation}
which is well-defined whenever $G \in C^1(\RR^n)$. As we will see in the next theorem the smoothness of $G$, which in our work is in general assumed to be $C(\RR^n)$, will have an impact on the rate of convergence.}

\begin{theorem}\label{WDESTTHM}
Let $\bar{u}(x,t)$ be a smooth solution of the equilibrium system \eqref{BLAW1D}, defined on $\RR^d \times [0,T]$. Let $\{u^{\eps}\}$ be a family of smooth solutions of the relaxation system \eqref{RXSYST1D} on $\RR \times [0,T]$. Suppose that both $\bar{u}$ and $u^{\eps}$ decay sufficiently fast at infinity and that:
\begin{itemize}
\item[(a1)]  (H1) holds true and  the positive definite, symmetric matrix $A$, is such that the subcharacteristic condition \eqref{AHYPSCENT1D} is valid.
\item[$(a2)$] For some $M>0$
\begin{equation*}
|D^2F(u)| \leq M, \quad |D^3\eta(u)| \leq M, \quad u\in \RR^n.
\end{equation*}
\item[(a3)] 
The conditions \eqref{WDSRC}, \eqref{POTN}  on the source term hold true.
 \end{itemize}

\par\smallskip

\noindent Then for all $t\in[0,T]$
\begin{equation}\label{WDEST}
  \Psi(t) + \eps \!\int_{\RR} R(u^{\eps}(x,t)) \SP dx \, \leq \, C \bigg( \Psi(0) + \eps \!\int_{\RR} R(u^{\eps}(x,0)) \SP dx  + \eps \bigg)
\end{equation}
with $C=C(\bar{u},\alpha,\beta,\nu,M,T)>0$ independent of $\eps$.

\par\bigskip

If, in addition, $G\in C^2(\RR^n)$ then for all $t\in[0,T]$
\begin{equation}\label{WDEST2}
  \Psi(t) + \eps \!\int_{\RR} \big[R^{rel}(\bar{u},u^{\eps}) \big](x,t) \, dx \, \leq \, C \bigg( \Psi(0) + \eps \!\int_{\RR} \big[R^{rel}(\bar{u},u^{\eps})\big](x,0) \, dx + \eps^2 \bigg)\,.
\end{equation}
\end{theorem}

\begin{proof} By \eqref{RXENTPROP}, \eqref{LISRC}, and \eqref{ERRT2} we obtain \begin{equation}\label{DESTforP3}\vspace{3pt}
\begin{aligned}
d_1    & \SP \leq \SP   (\D\eta(u^{\eps})-\D\eta(\bar{u})) (G(u^{\eps})-G(\bar{u})) \leq 0 \\
d_2  &\SP \leq \SP \alpha\BBN{G(\bar{u})}_{L^{\infty}}  |u^{\eps}-\bar{u}|^2 \\
\eps d_3 
& \,\leq \,-\eps \SP \del_t R(u^{\eps})  + \eps C_R (1+R(u^{\eps}))|\bar{u}_t| + \eps |G(\bar{u})||u^{\eps}_t-\bar{u}_t|.
\end{aligned}
\end{equation}
Then, combining \eqref{GFUNCEVOLPiii} and \eqref{DESTforP3} we obtain
\begin{equation}\label{INTESTWD}
\begin{aligned}
\frac{d}{dt} & \int_{\RR} \, \mathcal{G}(\bar{u},u^{\eps})\, dx + \eps \bar{c}
\int_{\RR} \, |u^{\eps}_x-\bar{u}_x|^2 + |u^{\eps}_t  - \bar{u}_t|^2 \, dx\\
& \leq \, C \Bigl( \Psi(t)+\eps^2 \Bigr) + 2 \SP \eps \alpha C_R\SP
\Bigl(\BBN{\bar{u}_t}_{L^1} + \BBN{\bar{u}_t}_{L^{\infty}} \int_{\RR}
R(u^{\eps})\,dx \Bigr)\\
& \,\,\quad -2 \eps \alpha \frac{d}{dt} \int_{\RR} \, R(u^{\eps})\, dx + \frac{2\SP\eps
\alpha^2}{\bar{c}} \int_{\RR} \, |G(\bar{u})|^2\, dx + \frac{\eps \bar{c}}{2} \int_{\RR} \,
|u^{\eps}_t - \bar{u}_t|^2 \,dx.
\end{aligned}
\end{equation}
Integrating the above inequality, and using \eqref{RXENTPROP}, \eqref{POTN}, \eqref{TP1a} and \eqref{PSIDEF},  we obtain
\begin{equation*}
\begin{aligned}
&\Psi(t) + \eps \int_{\RR} R(u^{\eps}(x,t)) \SP dx \\
&\quad \leq  C \bigg(\SP
\Psi(0) + \eps \int_{\RR} R(u^{\eps}(x,0)) \SP dx +  \int_{0}^t \Big\{\Psi(s) + \eps \int_{\RR} R(u^{\eps}(x,s)) \, dx \SP  \Big\} \SP ds +\eps t + \eps^2 t\bigg)
\end{aligned}
\end{equation*}
and conclude \eqref{WDEST} via the Gronwall lemma.

\par\medskip
{  \black
Suppose now that $G\in C^2(\RR^n)$. Then, using \eqref{POTN} and \eqref{RREL}, we obtain
\begin{equation}\label{DTRREL}
\begin{aligned}
  \del_t \big( R^{rel}(\bar{u},u^{\eps})\big) &= -G(u^{\eps})^{\top} u^{\eps}_t + G(\bar{u})^{\top}\bar{u}_t + G(\bar{u})^{\top}(u^{\eps}-\bar{u})_t \\ &\,\quad - (\D^2 R(\bar{u})\bar{u}_t)^{\top}(u^{\eps}-\bar{u})\\[3pt]
  &= (G(\bar{u})-G(u^{\eps}))^{\top}u^{\eps}_t - (\D^2 R(\bar{u})\bar{u}_t)^{\top}(u^{\eps}-\bar{u})\,.
\end{aligned}
\end{equation}
Then, using  \eqref{POTN}, \eqref{ERRT2}$_3$ and \eqref{DTRREL}, we obtain
\begin{equation}\label{d3viaRREL}
\begin{aligned}
  d_3 & = \bigl( G(u^{\eps})-G(\bar{u}) \bigr)^{\top} (u^{\eps}-\bar{u})_t\\[3pt]
      & = -\del_t \big( R^{rel}(\bar{u},u^{\eps}) \big) - (\D^2 R(\bar{u})\bar{u}_t)^{\top}(u^{\eps}-\bar{u}) - (G(u^{\eps})-G(\bar{u}))^{\top} \bar{u}_t\\[3pt]
            & = -\del_t \big( R^{rel}(\bar{u},u^{\eps}) \big) + \bar{u}_t^{\top}\big(\D R(u^{\eps})-\D R(\bar{u}) -\D^2 R(\bar{u})(u^{\eps}-\bar{u}) \big)\,.
\end{aligned}
\end{equation}
Then, combining \eqref{GFUNCEVOLPiii} with \eqref{DESTforP3}$_{1,2}$ and \eqref{d3viaRREL}, we obtain
\begin{equation}\label{INTESTWD2}
\begin{aligned}
     \frac{d}{dt} & \int_{\RR} \,\mathcal{G}(\bar{u},u^{\eps})\,dx + \eps \bar{c}
    \int_{\RR} \, |u^{\eps}_x-\bar{u}_x|^2 + |u^{\eps}_t-\bar{u}_t|^2 \,dx\\[2pt]
    & \leq \, C(T,\bar{u}) \Bigl( \Psi(t) + \eps^2\Bigr)  -2 \eps \alpha \frac{d}{dt}  \int_{\RR} R^{rel}(\bar{u},u^{\eps}) \, dx \\
   &\qquad + 2 \eps \alpha \int_{\RR}   \Big\{\bar{u}_t^{\top}\big(\D R(u^{\eps})-\D R(\bar{u}) -\D^2 R(\bar{u})(u^{\eps}-\bar{u}) \big)\Big\} \SP dx\,.
\end{aligned}
\end{equation}

\par\smallskip

{
Now, we estimate the last term on the right-hand side of \eqref{INTESTWD2}. By assumption on $\bar{u}$ there exists a compact set $\mathcal{K} \subset \RR^n$ such that
\begin{equation*}
\bar{u}(x,t) \in \mathcal{K}\,, \quad (x,t) \in \RR \times [0,T].
\end{equation*}
Thus, we can take large enough  $M>0$ such that
\begin{equation}\label{ballboundsol}
 \mathcal{K} \subset B_M\, \quad \mbox{and}  \quad |\bar{u}(x,t) -v | > 1 \quad \mbox{for all} \quad  v\in {B_M}^C\,, \, (x,t) \in \RR \times [0,T],
\end{equation}
where $B_M$ denotes a ball of radius $M$. Then, since $G=-\D R \in C^2(\RR^n)$, we obtain
\begin{equation}\label{ballboundG}
  \lambda^M_G := \sup_{v \in B_{M}} \Big( |R(v)|+|\D R(v)|+|\D^2 R(v)|+|\D^3 R(v)| \Big) \,  < \infty \,.
\end{equation}

\par\smallskip

Fix $(x,t)\in \RR \times [0,T]$. Suppose $u(x,t) \in \B_{M}$. Then by \eqref{ballboundG}
\begin{equation}\label{ballboundGREL}
  |\D R(u^{\eps}) - \D R(\bar{u}) - \D^2 R(\bar{u})(u^{\eps}-\bar{u})| \leq \lambda^M_{G} |u^{\eps}-\bar{u}|^2\,.
\end{equation}

\par\smallskip

Suppose now that $u^{\eps}(x,t) \in {\B_{M}}^C$. Then, by \eqref{ballboundsol} we have $|u^{\eps}(x,t)-\bar{u}(x,t)|> 1$. Hence
\begin{equation}\label{compballound1}
\begin{aligned}
  0 &\leq R(u^{\eps}(x,t)) = R^{rel}(\bar{u},u^{\eps}) + R(\bar{u})+\D R(\bar{u})(u-\bar{u})\\
  & \leq \max(R^{rel},|u^{\eps}-\bar{u}|^2)\Big(1 +  2\lambda^M_G \Big)\,.
  \end{aligned}
\end{equation}
Then, using \eqref{ballboundG}, \eqref{compballound1} and \eqref{POTN}, we obtain
\begin{equation}\label{compballound2}
\begin{aligned}
|\D R(u^{\eps}) &- \D R(\bar{u}) - \D^2 R(\bar{u})(u^{\eps}-\bar{u})| \\
&\, \leq \, 2\max(R(u^{\eps}),|u^{\eps}-\bar{u}|^2) \bigg( \frac{|\D R(u^{\eps})|}{R(u^{\eps})+1} + \lambda_G^M \bigg) \\
&\, \leq \, 2\max(R(u^{\eps}),|u^{\eps}-\bar{u}|^2) \big( C_R + \lambda_G^M \big)\\
& \leq C \max(R^{rel},|u^{\eps}-\bar{u}|^2)
\end{aligned}
\end{equation}
for some $C>0$ that depends on $\lambda_G^M$ and $C_R$.

\par\smallskip

Since $(x,t)$ is arbitrarily chosen, combining  \eqref{ballboundGREL} and \eqref{compballound2}, we conclude
\begin{equation}\label{GRELEST}
\begin{aligned}
|\D R(u^{\eps}) &- \D R(\bar{u}) - \D^2 R(\bar{u})(u^{\eps}-\bar{u})| \,  \leq \, C \max(R^{rel}(u^{\eps}),|  u^{\eps}-\bar{u}|^2)
\end{aligned}
\end{equation}
for all $(x,t) \in \RR \times [0,T]$. Thus
\begin{equation}\label{INTESTGREL}
\begin{aligned}
&2 \eps \alpha \!\int_{\RR} \Big\{\bar{u}_t^{\top}\big(\D R(u^{\eps})-\D R(\bar{u}) -\D^2 R(\bar{u})(u^{\eps}-\bar{u}) \big)\Big\} \SP dx \, \\
&\qquad\qquad\qquad \leq \, C \bigg(\eps \int_{\RR} \SP R^{rel}(\bar{u},u^{\eps}) \SP dx+\eps \int_{\RR}|\bar{u}-u^{\eps}|^2 \SP dx\bigg)\,.
\end{aligned}
\end{equation}
Then, integrating the inequality \eqref{INTESTWD2} in time, and using the \eqref{TP1a}, \eqref{PSIDEF},  and \eqref{INTESTGREL},  we obtain
\begin{equation}\label{INTESTWD2FIN}
\begin{aligned}
&      \Psi(t) + \eps \int_{\RR} \big[R^{rel}(\bar{u},u^{\eps})\big](x,t) \, dx \\ &\leq \, C \bigg( \Psi(0)+ \eps \int_{\RR} \big[R^{rel}(\bar{u},u^{\eps})\big](x,0) \SP dx  \\
 &\qquad\quad + \int_0^{\tau} \Bigl\{ \Psi(\tau) + \eps \int_{\RR} \big[R^{rel}(\bar{u},u^{\eps})\big](x,\tau) \SP dx \Bigr\} \SP d\tau + \eps^2 \bigg)
\end{aligned}
\end{equation}
which along with the Gronwall lemma implies \eqref{WDEST}.
}
}
\end{proof}

\subsection{Error estimates. General source $G(u)$}
We now consider the source term that satisfies the alternative hypothesis \eqref{LISRC}.
\begin{theorem}\label{gen-source}
 Let $\bar{u}(x,t)$ be a smooth solution of the equilibrium system \eqref{BLAW1D}, defined on $\RR\times [0,T]$. Let $\{u^{\eps}\}$ be a family of smooth solutions of the relaxation system \eqref{RXSYST1D} on $\RR \times [0,T]$. Suppose that both $\bar{u}$ and $u^{\eps}$ decay sufficiently fast at infinity and that:
\begin{itemize}\vspace{3pt}
\item[(a1)]  (H1) holds true and  the positive definite, symmetric matrix $A$, is such that the subcharacteristic condition \eqref{AHYPSCENT1D} is valid.
\item[$(a2)$] For some $M>0$
\begin{equation*}
|\D^2F(u)| \leq M, \quad |\D^3\eta(u)| \leq M, \quad u\in \RR^n.
\end{equation*}
\item[(a3)]  The condition (H3-b) on the source term holds true.
\end{itemize}
Then,
\begin{equation}\label{ESTG}\vspace{2pt}
    \Psi(t) \, \leq \, C \SP \bigl(\Psi(0)+\eps^2\bigr), \quad
    t\in[0,T]
\end{equation}
with $\Psi$ defined in \eqref{PSIDEF} and $C=C(\bar{u},\alpha,\beta,\nu,M,T)>0$ independent of $\eps$.
\end{theorem}

\begin{proof}
Let $\mathcal{A}\subset \RR^n$ denote a set such that $\bar{u}(x,t) \in \mathcal{A}$ for every $(x,t)$. Then \eqref{RXENTPROP}, \eqref{WDSRC}, \eqref{POTN} and \eqref{ERRT2} imply
\begin{equation}\label{DESTforP4}
\begin{aligned}
d_1    & \SP \leq \SP |u^{\eps}-\bar{u}| |G(u^{\eps})-G(\bar{u})| \leq L_{\mathcal{A}}|u^{\eps}-\bar{u}|^{2}  \\
d_2  &\SP \leq \SP \alpha\BBN{G(\bar{u})}_{L^{\infty}}  |u^{\eps}-\bar{u}|^2 \\
\eps d_3 & \SP \leq  \SP \eps L_{\mathcal{A}} |u^{\eps}-\bar{u}| |(u^{\eps}-\bar{u})_t| \leq L_{\mathcal{A}}^2 |u^{\eps}-\bar{u}|^2 + \eps^2|(u^{\eps}-\bar{u})_t|^2\,.
\end{aligned}
\end{equation}
Then, combining \eqref{PSIDEF}, \eqref{GFUNCEVOLPiii} and \eqref{DESTforP4},  we obtain
\begin{equation*}
\begin{aligned}
\frac{d}{dt} & \int_{\RR} \, \mathcal{G}(\bar{u},u^{\eps})\, dx + \eps \bar{c}
\int_{\RR} \, |u^{\eps}_x-\bar{u}_x|^2 + |u^{\eps}_t  - \bar{u}_t|^2 \, dx \leq \, C \Big( \Psi(t)+ \eps^ 2\Big)\,. \end{aligned}
\end{equation*}
Integrating the above inequality, and using \eqref{TP1a}, we obtain
\begin{equation*}
\begin{aligned}
\Psi(t) \leq   C \Big(\SP
\Psi(0) + \int_{0}^t \SP \Psi(s) \,ds + \eps^2 \Big)
\end{aligned}
\end{equation*}
and conclude \eqref{ESTG} via the Gronwall lemma.
\end{proof}

\section{Multidimensional case} \label{S6}

In this section we state our results for multidimensional systems. The proofs of these theorems follow similar line of argument as the ones presented in the earlier parts of this article, and therefore are here omitted.

We define
\begin{align}
    \varphi(t)&:=\int_{\RR^d} |u^{\eps}|^2 + \eps^2|\D u^{\eps}|^2 + \eps^2|u^{\eps}_t|^2 \,dx \label{PHIDEFMD} \\
    \Psi(t)&:=\int_{\RR} \, |\bar{u}-u^{\eps}|^2 + \eps^2|\D \bar{u}-\D u^{\eps}|^2 + \eps^2|(\bar{u}-u^{\eps})_t|^2\,dx\, \label{PSEDEFMD}
\end{align}
which are used in the next two subsections.

\subsection{Systems with weakly dissipative source $G$, $d \geq 1$}

The first result is the analog of Proposition \ref{w-diss2} on stability.
\begin{proposition}\label{w-dissMD}

Let $\{u^{\eps}(x,t)\}$ be a family of smooth solutions to the system \eqref{RXBLAWMDIM} on $\RR^d \times [0,T]$. Suppose that $u \equiv u^{\eps}$ decays fast at infinity and that:
\begin{itemize}
\item[(a1)]  \eqref{RXENTPROP} holds true, and the positive definite, symmetric matrices $A_j$, $j=1,\dots,d$ are such that the subcharacteristic condition \eqref{AHYPSCENTMD} is valid.
\item[(a2)]

The conditions \eqref{WDSRC} and  \eqref{POTN} for the source $G$ hold true.
 \end{itemize}

\noindent Then for all $t\in[0,T]$
\begin{equation}\label{PHIESTENTWMD}
  \varphi(t) + \eps \!\int_{\RR^d} R(u^{\eps}(x,t)) \,dx + \int_{0}^{t}\int_{\RR^d} |\D{\eta}(u)G(u)| \SP dx dt \, \leq \, C \Big( \varphi(0) + \eps \! \int_{\RR^d} R(u^{\eps}(x,0)) \,dx \Big)
\end{equation}
with $\varphi(t)$ defined in \eqref{PHIDEFMD}, $R(u)$ defined in \eqref{POTN}  and $C=C(A,\alpha,\beta)>0$ independent of $\eps$,  $T$.
\end{proposition}

\par\smallskip

The next theorem is the analog of compactness Theorem \ref{COMP-LIPS-1D}
\begin{theorem}
Let $d \geq 1$. Let $\{u_\eps \}$ be a family of smooth solutions of \eqref{RXBLAW1D}  on $\RR^d  \times [0, T]$ emanating from smooth initial data. The family $\{u_\eps\}$ is assumed to decay fast at infinity. Let the hypotheses of Proposition \ref{w-dissMD} remain valid so that the stability estimate \eqref{PHIESTENTWMD}  holds true with  $\eta-q$ entropy-entropy flux pair satisfying \eqref{RXENTPROP}, and $A_j$, $j=1,\dots,d,$ symmetric, positive-definite matrices subject to \eqref{AHYPSCENTMD}.
Then for an entropy pair $\bar{\eta}-\bar{q}$ satisfying the growth conditions
\begin{equation*}\|\bar{\eta}\|_{L^{\infty}}, \|\bar{q}\|_{L^{\infty}}, \|\D\bar{\eta}\|_{L^{\infty}}, \|\D^2\bar{\eta}\|_{L^{\infty}} \le C
\end{equation*}
and
\begin{equation*}
  |\D \bar{\eta}(v) G(v)| \leq C\big(M-\D \eta(v) G(v)\big)\,, \quad \forall v\in\RR^n
\end{equation*}
the family
\begin{equation*}
\Big\{\del_t \bar{\eta}(u^{\eps}) +  \sum_{j=1}^d \del_{x_j} \bar{q}_j(u^{\eps})\Big\}_{\eps}\,\, \mbox{lies in a compact set of }\,\, H^{-1}_{loc}(\RR^d \times [0,T]).
\end{equation*}
\end{theorem}

\par\smallskip

The next theorem is the analog of Theorem \ref{WDESTTHM} on convergence.
\begin{theorem}\label{WDESTTHMMD}
Let $\bar{u}(x,t)$ be a smooth solution of the system \eqref{BLAW1D}, defined on $\RR^d \times [0,T]$. Let $\{u^{\eps}\}$ be a family of smooth solutions of the relaxation system \eqref{RXSYST1D} on $\RR^d \times [0,T]$. Suppose that both $\bar{u}$ and $u^{\eps}$ decay sufficiently fast at infinity and that:
\begin{itemize}
\item[(a1)]  (H1) holds true and  the positive definite, symmetric matrices $A_j$, $j=1,\dots,d$ are such that the subcharacteristic condition \eqref{AHYPSCENTMD} is valid.
\item[$(a2)$] For some $M>0$
\begin{equation*}
|D^2 F_j(u)| \leq M, \; j=1,\dots, d, \quad |D^3\eta(u)| \leq M, \quad u\in \RR^n.
\end{equation*}
\item[(a3)] 
The conditions \eqref{WDSRC} and \eqref{POTN}  for the source term $G$ hold true.
 \end{itemize}

\noindent Then for all $t\in[0,T]$
\begin{equation*}
  \Psi(t) + \eps \!\int_{\RR^d} R(u^{\eps}(x,t)) \SP dx \, \leq \, C \bigg( \Psi(0) + \eps \!\int_{\RR^d} R(u^{\eps}(x,0)) \SP dx  + \eps \bigg)
\end{equation*}
with $\Psi$ defined in \eqref{PSEDEFMD}, $R(u)$ defined in \eqref{POTN} and $C=C(\bar{u},\alpha,\beta,\nu,M,T)>0$ independent of $\eps$.

\par\bigskip

If, in addition, $G\in C^2(\RR^n)$ then for all $t\in[0,T]$
\begin{equation*}
  \Psi(t) + \eps \!\int_{\RR^d} \big[R^{rel}(\bar{u},u^{\eps}) \big](x,t) \, dx \, \leq \, C \bigg( \Psi(0) + \eps \!\int_{\RR^d} \big[R^{rel}(\bar{u},u^{\eps})\big](x,0) \, dx + \eps^2 \bigg)
\end{equation*}
where
\begin{equation*}
   R^{rel}(\bar{u},u^{\eps}):= R(u^{\eps}) - R(\bar{u})  - \D R(\bar{u})(u-\bar{u}) \geq 0 \,.
\end{equation*}
\end{theorem}


\subsection{Systems with general source $G$, $d \geq 1$}

\begin{proposition}\label{Lgen-source-MD}
Let $\{u^{\eps}(x,t)\}$ be a family of smooth solutions to \eqref{RXBLAWMDIM} on $\RR^d \times [0,T]$. Suppose that  $u \equiv u^{\eps}$ decays fast at infinity and that:

\begin{itemize}
\item[(a1)]  (H1) holds true, and  positive definite, symmetric matrices $A_j$, $j=1,\dots,d$ are such that the subcharacteristic condition \eqref{AHYPSCENTMD} is valid.
\item[(a2)] The condition \eqref{LISRC} on the source $G$ holds true.
 \end{itemize}

\noindent Then
\begin{equation}\label{PHIESTGSRCMD}
  \varphi(t)  \SP \leq \SP  C\SP \varphi(0) \,, \quad t\in[0,T]
\end{equation}
with $\varphi$ defined in \eqref{PHIDEFMD} and $C=C(A,\alpha,\beta,T,L)>0$ independent of $\eps$ and $T$.
\end{proposition}

\par\medskip

\begin{theorem}
Let $d \geq 1$. Let $\{u_\eps \}$ be a family of smooth solutions of \eqref{RXBLAW1D}  on $\RR^d  \times [0, T]$ emanating from smooth initial data. The family $\{u_\eps\}$ is assumed to decay fast at infinity. Let the hypotheses of Proposition \ref{Lgen-source-MD} remain valid so that the stability estimate \eqref{PHIESTGSRCMD}  holds true with  $\eta-q$ entropy-entropy flux pair satisfying \eqref{RXENTPROP}, and $A_j$, $j=1,\dots,d,$ symmetric, positive-definite matrices subject to \eqref{AHYPSCENTMD}.
Then for an entropy pairs $\bar{\eta}-\bar{q}$ satisfying
\begin{equation*}\|\bar{\eta}\|_{L^{\infty}}, \|\bar{q}\|_{L^{\infty}}, \|\D\bar{\eta}\|_{L^{\infty}}, \|\D^2\bar{\eta}\|_{L^{\infty}} \le C
\end{equation*}
the family
\begin{equation*}
\Big\{\del_t \bar{\eta}(u^{\eps}) +  \sum_{j=1}^d \del_{x_j} \bar{q}_j(u^{\eps})\Big\}_{\eps}\,\, \mbox{lies in a compact set of }\,\, H^{-1}_{loc}(\RR^d \times [0,T]).
\end{equation*}
\end{theorem}

\par\smallskip

\begin{theorem}\label{gen-source-MD}
 Let $\bar{u}(x,t)$ be a smooth solution of the system \eqref{BLAWMDIM}, defined on $\RR\times [0,T]$. Let $\{u^{\eps}\}$ be a family of smooth solutions of the relaxation system \eqref{RXSYSTMDIM} on $\RR^d \times [0,T]$. Suppose that both $\bar{u}$ and $u^{\eps}$ decay sufficiently fast at infinity and that:
\begin{itemize}
\item[(a1)]  (H1) holds true and  the positive definite, symmetric matrix $A$, is such that the subcharacteristic condition \eqref{AHYPSCENTMD}  is valid.
\item[$(a2)$] For some $M>0$
\begin{equation*}
|\D^2 F_j(u)| \leq M,\; j=1,\dots, d \,,\quad |\D^3\eta(u)| \leq M, \quad u\in \RR^n.
\end{equation*}
\item[(a3)]  The condition (H3-b) for the source term $G$ holds true.
\end{itemize}
Then
\begin{equation*}
    \Psi(t) \, \leq \, C \SP \bigl(\Psi(0)+\eps^2\bigr), \quad
    t\in[0,T]
\end{equation*}
with $\Psi$ defined in \eqref{PSEDEFMD} and $C=C(\bar{u},\alpha,\beta,\nu,M,T)>0$ independent of $\eps$.
\end{theorem}

\section{An alternative relaxation model} \label{S7}

In this section we consider an alternative relaxation model for the system of hyperbolic balance laws \eqref{BLAWMDIM} given by
\begin{equation}\label{RXSYSTMDIMALT}
\left\{
\begin{aligned}
\del_t u + \sum_{j=1}^d \del_{x_j} v_j &= G(u) \\
\del_t v_i + A_i \del_{x_i} u &= -\frac{1}{\eps} \bigl( v_i-F_i(u))\bigr), \quad i=1,\dots,d
\end{aligned} \right.
\end{equation}
with $u, v_i\in \RR^n$ and  $A_i$ symmetric, positive definite matrix. Excluding $v_i$ from the equation \eqref{RXSYSTMDIMALT}$_1$  and assuming $G \in C^1(\RR^n)$, we obtain
\begin{equation}\label{RXBLAWMDIMALT}
\begin{aligned}
\del_t u + \sum_{j=1}^d \del_{x_j} F_j(u) = G(u) + \eps \SP \del_t ( G(u)) + \eps \Bigl(\sum_{j=1}^d A_j u_{{x_j}
{x_j}} - u_{tt}\Bigr)
\end{aligned}
\end{equation}
that approximates the system of balance laws \eqref{BLAWMDIM}.
We remark that the treatment of such relaxation systems presents several challenges. More specifically, {the time derivative} of the source term appear in the energy functional posing enormous difficulties  in the analysis,  an additional hypothesis is required for the establishment of stability (see (H7), Section \ref{S6}), the issue of compactness is problematic.
In the next two subsections we will study the stability and compactness properties of solutions to \eqref{RXBLAWMDIMALT} in order to point out the advantages of the relaxation model \eqref{RXSYSTMDIM}. We restrict the analysis to weakly dissipative systems of hyperbolic balance laws equipped with strictly convex entropy, $d=1$.

\subsection{Weakly dissipative systems with strictly convex entropy, $d=1$}

The system  \eqref{RXBLAWMDIMALT} for $d=1$ reads
\begin{equation}\label{RXBLAW1DALT}
\partial_{t}u+\partial_{x}F(u)=G(u)+\eps \SP \del_t ( G(u))+\eps(Au_{xx}-u_{tt})
\end{equation}
where $A$ is symmetric, positive definite matrix. Following the arguments of Lemma \ref{energy} we obtain:
\begin{lemma} \label{energyalt}
Suppose $u \equiv u^{\eps}(x,t)$ is a smooth solution to the equation \eqref{RXBLAW1DALT} on $\RR \times [0,T]$, $\eta-q$ is
the entropy-entropy flux pair of \eqref{BLAW1D} and $\bar{\alpha}\in\RR$ is fixed. Then, the following energy identity holds
\begin{equation}\label{STABIDSCENT1DALT}
\begin{aligned}
&\partial_{t}\bigg[\eta(u+\eps
u_{t})+\tfrac{1}{2}\eps^2\bar{\alpha} \SP |u_{t}|^2+\eps^{2}\bar{\alpha}\SP
u_x ^{\top} A u_x\\
&\qquad+\eps^2u_{t}^{\top} \Big(\,\tfrac{1}{2}\bar{\alpha} {\bf I}-\!
\int_{0}^{1}\!\int_{0}^{s} \D^2{\eta}(u+\eps\tau u_{t})\,d\tau ds\Big)u_{t} \bigg]+\partial_{x} q(u)\\[3pt]
&\quad+\eps\bar{\alpha}\big|u_{t}+\D F(u)u_{x}\big|^2+\eps u_{t}^{\top} \Big(\bar{\alpha} I-\D^2{\eta}(u)\Big)u_{t}\\[3pt]
&\quad+\eps u_{x} ^{\top} \Big(\D^2{\eta}(u)A-\bar{\alpha} \D F(u)^{\top}\D F(u)\Big)u_{x}\\[4pt]
&= \, \partial_{x}\Big(\eps \D {\eta}(u) Au_{x}+2\eps^2 \bar{\alpha} \SP u_{t}^{\top}
Au_{x}\Big)\\
&\quad +\D{\eta}(u) \Big( G(u)+ \eps\del_t(G(u))\Big)+ 2\SP\eps  \bar{\alpha} u_{t}^{\top} \Big( G(u) + \eps\del_t(G(u)\Big).
\end{aligned}
\end{equation}
\end{lemma}

\par\smallskip

{  \black
In our further analysis we will employ the following elementary lemma.
\begin{lemma}[\bf Weak dissipation of a gradient]\label{D2RPOS}
Suppose $\eta(u)$ satisfies \eqref{RXENTPROP}, and $G(u)\in C^1$ satisfies \eqref{WDSRC}, \eqref{POTN}. Then
\begin{equation}\label{D2RPOSPROP}
- \D G (u) =  \D ^2 R(u) \SP \geq 0\, \quad \mbox{for all} \quad u \in \RR^n\,.
\end{equation}
In addition, if \SP $\tilde{\eta}(u):\RR^n \to \RR$ \SP satisfies \SP $\D^2 \tilde{\eta}(u) \geq 0$,\SP\SP then
\begin{equation}\label{D2RWDPROP}
- \D^2 \tilde{\eta} (u) \D G(u)  = \D^2 \tilde{\eta} (u) \D^2 R(u) \geq 0 \quad \mbox{for all} \quad u \in \RR^n\,.
\end{equation}
\end{lemma}

\begin{proof}
From \eqref{WDSRC}, \eqref{POTN} it follows that
\begin{equation*}
- z^{\top} \D^2 \eta (u) \D G(u) \SP z = z^{\top} \D^2 \eta (u) \D^2 R(u) \SP z \geq 0 \quad \mbox{for all} \quad u, \SP z \in \RR^n\,.
\end{equation*}

\par\smallskip

 Fix $u\in\RR^n$. Let $(\lambda,v)$ be an eignepair of $\D^2 R (u)$. Then the above inequality implies
\begin{equation*}
0 \leq  v^{\top} \D^2 \eta (u) \D^2 R(u) \SP v = \lambda \big(v^{\top} \D^2 \eta (u) v \big).
\end{equation*}
Recalling that $\D^2 \eta(u)>0$ and $v \neq 0$, we conclude that $\lambda \geq 0$. Hence all eigenvalues of $\D^2 R(u)$ are nonnegative. Since $\D^2 R(u)$ is a symmetric matrix this is equivalent to $\D^2 R (u) \geq 0$.

\par

Next, suppose $\tilde{\eta} (u)$ is a strictly convex function. Then \eqref{D2RWDPROP} follows from the fact that $\D^2 \tilde{\eta}$, $\D^2 R$ are symmetric, nonnegative definite matrices.
\end{proof}

{ \begin{remark}
Note that the requirement of uniform convexity for $\eta(u)$ in Lemma \ref{D2RPOS} can be replaced with strict convexity. Thus, when $G$ is a gradient, weak dissipation of $G$ with respect to some strictly convex entropy automatically yields weak dissipation with respect to any convex entropy, a property which is not in general satisfied.
\end{remark} }
}

To establish the stability of solutions to \eqref{RXBLAW1DALT} we will employ an additional hypothesis:
\begin{itemize}
\item  Suppose there exists $S:\RR^n\to\RR$ \SP such that
\begin{equation}\tag{H5}\label{ENTSRCPROP}
\begin{aligned}
\D \eta(u) \D G(u)&=-\D S(u), \quad u \in \RR^n\\
S(u) \geq S(0)&=0\,.
\end{aligned}
\end{equation}
\end{itemize}

\begin{remark}
We note that the hypothesis \eqref{ENTSRCPROP} is a severe assumption. It is motivated by the following observation. Suppose that the entropy $\eta(u)=\frac{1}{2} |u|^2$, and \eqref{WDSRC}, \eqref{POTN} hold true. Then,
\begin{equation*}
  \D \eta(u) \D G (u)= u^{\top} \D^2 R(u) =  - \D S(u) \quad \mbox{with} \quad S(u)=\D R(u)u - R(u)\,.
\end{equation*}
Moreover, by Lemma \ref{D2RPOS} it follows that $\D^2 R \geq 0$ and hence $S(u) \geq S(0)=0$.
\end{remark}

\par\smallskip

\begin{proposition}\label{w-diss2-alt}

Let $\{u^{\eps}(x,t)\}$ be a family of smooth solutions to the equation \eqref{RXBLAW1DALT} defined on $\RR \times [0,T]$. Suppose that $u \equiv u^{\eps}$ decays fast at infinity and that:

\begin{itemize}
\item[(a1)]  \eqref{RXENTPROP} holds true, and the positive definite, symmetric matrix $A$ is such that the subcharacteristic condition \eqref{AHYPSCENT1D} is valid.

\item[(a2)] The conditions \eqref{WDSRC},  \eqref{POTN}, and \eqref{ENTSRCPROP} for the source $G \in C^1(\RR^n)$ hold true.
 \end{itemize}
Then for all $t\in[0,T]$
\begin{equation}\label{PHIESTENTW1DALT}
\begin{aligned}
 \varphi(t) &+ \eps \int_{\RR} \Big\{S(u^{\eps}(x,t))+R(u^{\eps}(x,t))\Big\} \,dx  +  \int_{0}^{t}\int_{\RR} \SP |\D{\eta}(u)G(u)| \, dx dt  \\
 & \qquad\qquad \leq \, C \bigg( \varphi(0) + \eps \int_{\RR} \Big\{S(u^{\eps}(x,0))+R(u^{\eps}(x,0))\Big\} \SP dx \bigg)
\end{aligned}
\end{equation}
with $\varphi$ defined in \eqref{PHIDEF}
and $C=C(A, \alpha, \beta)>0$ independent of both  $\eps$ and $T$.
\end{proposition}

\begin{proof} Using \eqref{POTN}, \eqref{ENTSRCPROP} we rewrite the last four terms on the right-hand side of \eqref{STABIDSCENT1DALT}, with $\bar{\alpha}=\alpha$, as follows
\begin{equation}\label{WDSRCENTALT}
\begin{aligned}
&\Big(\D \eta(u) G(u)+\eps \D \eta(u)\SP\del_t (G(u))+2\SP\eps u_{t}^{\top}G(u)+2\SP\eps^2 \alpha  u_{t}^{\top}\del_t(G(u))\Big)\\
&= \Big(\D \eta(u)G(u)-2\eps^2 \alpha \SP u_t^{\top} \D^2 R(u) \SP u_t\Big) - \eps  \del_t \Big( S(u)+ 2\alpha R(u) \Big) =: I_1 + I_2.
\end{aligned}
\end{equation}
Since $G \in C^1$, by \eqref{RXENTPROP}, \eqref{WDSRC}, \eqref{POTN}  and Lemma \ref{D2RPOS}
\begin{equation}\label{I1b}
\alpha \SP u_t^{\top} \D^2 R(u) \SP u_t \geq 0\,.
\end{equation}
Thus, by \eqref{WDSRC} and \eqref{I1b} we conclude $I_1 \leq 0$. Then, integrating the identity \eqref{STABIDSCENT1DALT}, with $\bar{\alpha}=\alpha$, and employing hypotheses \eqref{RXENTPROP}, \eqref{AHYPSCENT1D}, \eqref{WDSRC}, \eqref{POTN}, \eqref{ENTSRCPROP} along with \eqref{ALPHADOM}, \eqref{PHIEQUIVENT}  and \eqref{WDSRCENTALT}, we conclude \eqref{PHIESTENTW1DALT}.
\end{proof}

\subsection{Compactness issues}

Let $\{u_\eps \}$ be a family of smooth solutions of \eqref{RXBLAW1DALT}  on $\RR  \times [0, T]$ emanating from smooth initial data. The family $\{u_\eps\}$ is assumed to decay fast at infinity. Let the hypotheses of Lemma \ref{w-diss2-alt} remain valid so that the stability estimate \eqref{PHIESTENTW1DALT}  holds true with  $\eta-q$ entropy-entropy flux pair satisfying \eqref{RXENTPROP}, and $A$ a symmetric, positive-definite matrix subject to \eqref{AHYPSCENT1D}. Let the entropy pairs $(\bar{\eta}, \bar{q})$ satisfy \eqref{entropy-hyp} and the growth condition holds \eqref{proj-growth}. We now show that, in general, one may not expect compactness from the family $\{\del_t \bar{\eta}(u^{\eps}) + \del_x \bar{q}(u^{\eps}) \}$.

\par\smallskip

We follow the arguments of Theorem \ref{COMP-WD-1D}. Starting from \eqref{RXBLAW1DALT} we obtain
\begin{equation*}
\begin{aligned}
\del_t \bar{\eta}(u^{\eps}) + \del_x \bar{q}(u^{\eps}) &= \eps \del_x \left( \D\bar{\eta}(u^{\eps}) A u^{\eps}_x \right) - \eps \del_t \left( \D\bar{\eta}(u^{\eps}) u^{\eps}_t\right) \\[1pt]
&\qquad -\eps {u^{\eps}_x} ^{\top} \D^2\bar{\eta}(u^{\eps}) A u^{\eps}_x + \eps {u^{\eps}_t} ^{\top} \D^2 \bar{\eta}(u^{\eps}) u^{\eps}_t\\[1pt]
 &\qquad + \D\bar{\eta}(u^{\eps}) G(u^{\eps}) + \eps \D\bar{\eta}(u^{\eps}) \D G(u^{\eps}) u^{\eps}_t\\[3pt]
&\quad:= I_1 + I_2+ I_3 + I_4 + I_5 + I_6.
\end{aligned}
\end{equation*}
As before from \eqref{entropy-hyp}, \eqref{proj-growth} and \eqref{PHIESTENTW1DALT}, it follows that the terms $I_1, I_2$ lie in  compact set of $H^{-1}$, and the terms $I_3, I_4, I_5$ are bounded in $L^1$. However, the term $I_6$ is,  in general, neither in a bounded set of the space of measures nor in a compact set of $H^{-1}$: for the former one must have a control of $\eps \D G$ in $L^2$, and for the latter a control of $\eps G$ in $L^2$, which follows from the relation
\begin{equation*}
I_6 =  \del_t \big(\D\bar{\eta}(u^{\eps}) \eps G(u^{\eps})\big) - \big( \D^{2}\bar{\eta}(u^{\eps}) u^{\eps}_t \big)^{\top} \eps G(u^{\eps})\,.
\end{equation*}
The stability estimate \eqref{PHIESTENTW1DALT}, however, does not provide such bounds. Thus, Murat's lemma is not applicable and the issue of compactness appears problematic.

\section{Applications} \label{S8}

\subsection{Elasticity system}
Consider the relaxation of the (isothermal/isentropic) elasticity system with a source term:
\begin{equation}\label{ELASTSYST}
\begin{aligned}
\left[\begin{aligned}
&u\\
&v\\
\end{aligned}\right]_t -
\left[\begin{aligned}
&v\\
\sigma&(u)\\
\end{aligned}\right]_x = G(u,v)=\left[\begin{aligned}
&0\\
g(&u,v)\\
\end{aligned}\right].
\end{aligned}
\end{equation}
In the context of gas flow, $u$ is specific volume (
$u = 1/\rho$), thus constrained by $u > 0.$ In the context of the thermoelastic bar, $u$ is the
{\em strain}, likewise constrained by $u > 0.$ Finally, in the context of shearing motion, $u$
is shearing, which may take both positive and negative values. In the gas case, it is
traditional to use the pressure $p= - \sigma,$ instead of $\sigma.$

In the present context, the stress $\sigma(u)$ is assumed to satisfy
\begin{equation}\label{GSTRESS}
\sigma(0)=0 \quad \mbox{and} \quad  0 < \gamma < \sigma'(u) < \Gamma\quad \mbox{for all} \quad u \in \RR^n.
\end{equation}
We assume that $g(u,v)$, with $g(0,0)=0$, satisfies one of the following:
\begin{itemize}
\item [$(i)$] {\it Either} $g$ is independent of $u$, that is $g(u,v)=g(v)$, and
satisfies
\begin{equation}\label{WKDSPELAST}
    \bigl(g(v)-g(\bar{v})\bigr)\bigl(v-\bar{v}\bigr) \leq 0, \quad \forall \SP
    v,\bar{v}\in \RR
\end{equation}
which corresponds to a {\it frictional damping}.
\end{itemize}
\begin{itemize}
\item [$(ii)$] {\it or} \SP for every compact set $\mathcal{A} \subset \RR^2$ there
exists $L_{\mathcal{A}}>0$ such that
\begin{equation}\label{SRCLIPELAST}
|g(u,v)-g(\bar{u},\bar{v})| \leq L_{\mathcal{A}}
\bigl(|u-\bar{u}|+|v-\bar{v}|\bigr)
\end{equation}
for all $\SP (u,v)\in \RR^2$, $(\bar{u},\bar{v})\in\mathcal{A}$.
\end{itemize}
The system \eqref{ELASTSYST} is equipped with the entropy - entropy flux pair
$\bar{\eta},\bar{q}$ given by
\begin{equation}\label{ELASTENTPAIR}
\eta(u,v) = \tfrac{1}{2}v^2 + \Sigma(u), \quad q(u,v) = -\sigma(u)v \quad
\mbox{with} \quad \Sigma(u) := \int^{u}_0 \sigma(\tau)\SP d\tau \SP.
\end{equation}

\par\smallskip

For the system \eqref{ELASTSYST} the second order relaxation system \eqref{RXBLAW1D} reads
\begin{equation}\label{ELASTSYSTRELAX}
\begin{aligned}
\left[\begin{aligned}
&u\\
&v\\
\end{aligned}\right]_t -
\left[\begin{aligned}
&v\\
\sigma&(u)\\
\end{aligned}\right]_x =\left[\begin{aligned}
&0\\
g(&u,v)\\
\end{aligned}\right] + \eps \bigg(A \left[\begin{aligned}
&u\\
&v\\
\end{aligned}\right]_{xx}
-
 \left[\begin{aligned}
&u\\
&v\\
\end{aligned}\right]_{tt}\bigg)\,.
\end{aligned}
\end{equation}
Observe that by \eqref{GSTRESS} and \eqref{ELASTENTPAIR} hypotheses \eqref{RXENTPROP}-\eqref{AHYPSCENT1D} are satisfied with
\begin{equation}\label{SUBCELAST}
  \alpha = \max(2\Gamma,1), \quad \beta=\min(\gamma,1), \quad A=2\alpha{\bf I}\,.
\end{equation}

\par

Suppose that \eqref{WKDSPELAST} holds true. Then
\begin{equation*}
  \big(\D \eta(u,v)-\eta(\bar{u},\bar{v})\big)  \big(G(u,v)-G(\bar{u},\bar{v})\big) = (v-\bar{v})(g(v)-g(\bar{v})) \leq 0
\end{equation*}
and hence $G$ satisfies \eqref{WDSRC}. Furthermore, setting
\begin{equation*}
  R(u,v):=-\int^v_0 g(\theta) \SP d\theta \quad \mbox{we get} \quad G(u,v) = -\D R (u,v)\,, \quad R(0,0)=0
\end{equation*}
which gives \eqref{POTN}. { Thus, one may employ Proposition \ref{w-diss2} to obtain the stability estimate \eqref{PHIESTENTW1D} and Theorem \ref{WDESTTHM} to obtain the error estimate \eqref{ELASTSYSTRELAX} that holds before the formation of shocks.

\par\smallskip

Similarly, suppose \eqref{SRCLIPELAST} holds true. Then clearly $G$ satisfies \eqref{LISRC} and one may use Proposition \ref{Lgen-source-2} and Theorem \ref{gen-source} to obtain the stability estimate \eqref{PHIESTSCENTG1D} and the error estimate \eqref{ELASTSYSTRELAX}, respectively.

\par\smallskip

The stability estimates \eqref{PHIESTENTW1D} and \eqref{PHIESTSCENTG1D}  suffice to apply the $L^p$ theory of compensated compactness. In the spirit of \cite[Theorem 1]{GT-2001} we prove the following convergence theorem.
\begin{theorem} Let $\sigma(u)$ satisfy \eqref{GSTRESS} and suppose also
\begin{equation}\label{SHCOND}
  (u-u_0)\SP g''(u) \neq 0 \;\; \mbox{for} \;\; u \neq u_0 \;\; \mbox{and} \;\; g'', \, g''' \in L^2 \bigcap L^{\infty}.
\end{equation}
Let $(u^{\eps},v^{\eps})$ be a family of smooth solutions of \eqref{ELASTSYST} defined on $\RR \times [0,T)$ emanating from smooth initial data subject to $\eps$-independent bounds
\begin{equation*}
  \varphi(0) = \int_{\RR} \Big({u_0^{\eps}}^2 + {v_0^{\eps}}^2 \Big) dx + \eps^2 \int_{\RR} \Big( |{u^{\eps}_0}_x|^2 + |{u_0^{\eps}}_t|^2  \,\Big) dx \, \leq \, C_0.
\end{equation*}
 Let $A$ be a symmetric, positive definite matrix satisfying \eqref{SUBCELAST} and let $g(u,v)$ satisfy  either \eqref{WKDSPELAST} or \eqref{SRCLIPELAST}. Then, along a subsequence if necessary,
  \begin{equation*}
   u^{\eps} \to u \,, \, v^{\eps} \to v\,, \;\; \mbox{a.e.} \;\; (x,t) \;\; \mbox{and} \;\; in \;\; L^p_{loc}(\RR \times (0,T)), \quad \mbox{for} \;\; p < 2.
 \end{equation*}
\end{theorem}

\begin{proof}
Let $(u^{\eps},v^{\eps})$ be a family of solutions to \eqref{ELASTSYSTRELAX}. The proof uses the theory of compensated compactness \cite{Tartar-1979}. Typically, in such proofs, the goal is to control the dissipation measure and to show that
\begin{equation}\label{cc-elast}
\Big\{\del_t \bar{\eta}(u^{\eps},v^{\eps}) +  \del_x \bar{q}(u^{\eps},v^{\eps})\Big\}_{\eps}\,\, \mbox{lies in a compact set of }\,\, H^{-1}_{loc}(\RR \times (0,T))
\end{equation}
for a class of entropy-entropy flux pairs $\bar{\eta}-\bar{q}$ for the equations of elasticity. In the presence of uniform $L^{\infty}$-bounds, the theorem of DiPerna \cite{DiPerna83} guarantees compactness of approximate solutions and implies that, along a subsequence, $u^{\eps} \to u$ and $v^{\eps} \to v$ a.e. $(x,t)$.

\par\smallskip

In the present case $L^1$-estimates are only available in the special case that $A$ is a multiple of the identity matrix (see \cite{Serre-2000}) and, in view of \eqref{PHIESTENTW1D} and \eqref{PHIESTSCENTG1D}, the natural stability framework is in the energy norm. Convergence of viscosity approximations to the equations of elastodynamics in the energy framework is carried out in Shearer \cite{Sh} (for the genuine-nonlinear case) and Serre-Shearer \cite{SSh} (for loss of genuine nonlinearity
at one point). In \cite{Sh} two classes of entropies, with growth controlled
by the wave-speeds at infinity, are constructed (\cite[Lemma 2]{Sh}) for which Tartar's commutation relation is justified (see \cite[Lemma 2]{Sh}) and are used to show that the support of the (generalized) Young measure is a point mass (\cite[Lemma 7, Theorem 1-(iii)]{Sh}). When $\sigma(u)$ has one
inflection point, the reduction of the Young measure is performed in \cite[Lemma 3]{SSh} and \cite[Section 5]{SSh}.

To ensure the dissipation estimate, we employ the growth assumption \eqref{GSTRESS}, the subcharacteristic condition \eqref{SUBCELAST}$_3$ and the assumption on the source \eqref{WKDSPELAST}, \eqref{SRCLIPELAST}. Then it suffices to establish \eqref{cc-elast} for entropy pairs $\bar{\eta}-\bar{q}$ satisfying
\begin{equation}\label{entropy-hyp-elast}
\|\bar{\eta}\|_{L^{\infty}}\,, \|\bar{q}\|_{L^{\infty}}\,, \|\D\bar{\eta}\|_{L^{\infty}}\,, \|\D^2\bar{\eta}\|_{L^{\infty}} \le C \,.
\end{equation}
This class of entropy pairs contains (under the assumption \eqref{GSTRESS}) the test-pairs that are used in \cite{Sh,SSh} in order to prove the reduction of the generalized Young measure to a point mass and to show strong convergence in $L^p_{loc}$ for $p < 2$. Hypothesis \eqref{SHCOND}
reflects the assumptions needed in those works. To complete the proof, we show that \eqref{cc-elast} holds for entropy-entropy flux pair $\bar{\eta}-\bar{q}$ satisfying \eqref{entropy-hyp-elast}.

First suppose that  $g(u,v)$ satisfies \eqref{WKDSPELAST}. Then by \eqref{ELASTENTPAIR}
\begin{equation}\label{proj-growth-elast}
\begin{aligned}
 & |\D \bar{\eta}(u,v) G(u,v)| =  |\bar{\eta}_{v}(u,v) g(v)| \leq C\big(M +\D \eta G(u,v)\big)\,, \quad M=\max_{|v|\leq 1}{|g(v)|}\,.
\end{aligned}
\end{equation}
The inequality \eqref{proj-growth-elast} is the analog of the condition \eqref{proj-growth} used in Theorem \ref{COMP-WD-1D} (in which case, the condition \eqref{proj-growth} is {\it automatically} satisfied  for the elasticity system  \eqref{ELASTSYST}). Then by \eqref{entropy-hyp-elast}, \eqref{proj-growth-elast}, and Theorem \ref{COMP-WD-1D} we conclude \eqref{cc-elast}.

\par\smallskip

Suppose now that $g(u,v)$ satisfies \eqref{SRCLIPELAST}. Then $G$ satisfies \eqref{LISRC} and hence from \eqref{entropy-hyp-elast}, \eqref{proj-growth-elast}, and Theorem \ref{COMP-LIPS-1D} we obtain \eqref{cc-elast}.

\end{proof}
}

\par\smallskip

We remark that the alternative relaxation system \eqref{RXBLAW1DALT} reads
\begin{equation}\label{ELASTSYSTRELAXALT}
\begin{aligned}
\left[\begin{aligned}
&u\\
&v\\
\end{aligned}\right]_t -
\left[\begin{aligned}
&v\\
\sigma&(u)\\
\end{aligned}\right]_x =\left[\begin{aligned}
&0\\
g(&u,v)\\
\end{aligned}\right] + \eps \left[\begin{aligned}
&0\\
g(&u,v)\\
\end{aligned}\right]_t+ \eps \bigg(A \left[\begin{aligned}
&u\\
&v\\
\end{aligned}\right]_{xx}
-
 \left[\begin{aligned}
&u\\
&v\\
\end{aligned}\right]_{tt}\bigg)\,.
\end{aligned}
\end{equation}
Then setting
\begin{equation*}
S(u,v):= -g(v)v-R(u,v) \;\;\; \mbox{we get} \;\;\; \D\eta(u,v) \D G(u,v) = vg'(v) = -\D S(u,v), \; S(0,0)=0
\end{equation*}
which gives \eqref{ENTSRCPROP}. Thus, the stability of solutions to \eqref{ELASTSYSTRELAXALT} follows from Proposition \ref{w-diss2-alt}.

\subsection{Isentropic  combustion model}
{
The governing equations for chemical reaction from {\it unburnt} gases to {\it burnt} gases in certain physical regimes (in Lagrangian coordinates) read \cite{CHT03}:
\begin{equation}\label{COMBMCONS}
\begin{aligned}
&\del_t v - \del_x u  = 0\\
&\del_t u + \del_x ( P(v,s,Z) ) = 0\\
&\del_t \bigl( E(v,s,Z) + \tfrac{1}{2} u^2 + qZ \bigr)_t + \del_x(u P(v,s,Z)) = r\\
&\del_t Z + K \varphi(\Theta(v,s,Z)) Z = 0
\end{aligned}
\end{equation}
The state of the gas is characterized by the macroscopic variables: the specific volume $v(x,t),$  the velocity field $u(x,t),$ the entropy $s(x,t)$ and the mass fraction of the
reactant $Z(x,t),$ whereas the physical properties of the material are reflected through
appropriate constitutive relations which relate the pressure $P(v,s,Z)$, internal energy
$E(v,s,Z)$ with the macroscopic variables. Here, and in what follows, $q$ represents the
difference in the heats between the reactant and the product, $K$ denotes the rate of the
reactant, whereas $\varphi(\theta) \geqslant 0$ is the reaction function. The function
$r(x,t)$ represents a source term (additional radiating heat density).

\par\smallskip
In this section we address the problem of
relaxation to the {\it isentropic combustion} model
\begin{equation}\label{ISCOMB}
\begin{aligned}
\left[\begin{aligned}
&v\\
&u\\
&Z\\
\end{aligned}\right]_t +
\left[\begin{aligned}
-u&\\
P(v,&Z) \\
0&\\
\end{aligned}\right]_x   =
\left[\begin{aligned}
\!0&\\
\!0&\\
\!-K\varphi(&\Theta(v,Z))\\
\end{aligned}\right]
\end{aligned}
\end{equation}
that arises naturally from the system \eqref{COMBMCONS}  by externally regulating $r$ to ensure
$s=s_0$ \cite{Dafermos79}; in the sequel we suppress the variable $s$ and use the notation
\begin{equation*}
P(v,Z):=P(v,s_0,Z), \quad \Theta(v,Z):=\Theta(v,s_0,Z).
\end{equation*}}

\par\smallskip

We impose the following requirements on $P,\Theta$ (see \cite{MT-2013a} for the motivation):
%
%
%
%
%
\begin{itemize}
\item[$(a1)$] Motivated by the physical property  $\del_v P <0$ we assume that
\begin{equation*}
0 \SP < \SP \gamma \SP < -\del_v P(v,Z)  \SP < \SP \Gamma, \quad v\in\RR,\SP Z\in[0,1].
\end{equation*}

\item[$(a2)$] There exists $\bar{C}>0$ \SP such that
\begin{equation*}
\quad \Bigl|\int_{0}^{v} P_{ZZ}(\tau,Z)\SP d\tau \Bigr| < \bar{C}, \,\,\quad |\del_Z
P(v,Z)| < \bar{C}, \quad v\in\RR,\SP Z\in[0,1].
\end{equation*}

\item[$(a3)$] The composition $\varphi \circ \Theta$ of the rate and constitutive
temperature functions satisfies for some $L>0$
\begin{equation}\label{COMPOSLIP}
\bigl|\varphi(\Theta(v,Z))-\varphi(\Theta(\bar{v},\bar{Z}))\bigr| \, \leq \, L
|(v,Z)-(\bar{v},\bar{Z})|
\end{equation}
for all $(v,Z), (\bar{v},\bar{Z})\in \RR\times[0,1].$

\end{itemize}

\par\smallskip

Under $(a1)$-$(a3)$ the system \eqref{ISCOMB} admits an entropy-entropy flux pair
$\bar{\eta},\bar{q}$ of the form:
\begin{equation*}
\begin{aligned}
\eta(v,u,Z) = \frac{1}{2}u^2 - \biggl(\int^{v}_0 \SP P(\tau,Z) \SP d\tau \biggr) +
B(Z)\,,
\quad q(v,u,Z) = P(v,Z)u
\end{aligned}
\end{equation*}
with $B(Z)$ an {\it arbitrary function}. To ensure the convexity of $\eta$ we assume, in addition to (a1)-(a3), that
\begin{equation*}
  B''(Z) > \Big(1+\frac{2}{\Gamma}\bar{C}^2+\bar{C}\Big), \quad Z \in [0,1]\,,
\end{equation*}
in which case
\begin{equation}\label{COMBCONV}
0<  \min\big(\frac{\gamma}{2},1\big) \, \leq \, \D ^2 \eta(v,u,Z) \leq \max\Big(1, \Gamma+\bar{C}, \frac{2}{\Gamma}\bar{C}^2+2\bar{C} \Big).
\end{equation}

For the system \eqref{ISCOMB} the second order relaxation system \eqref{RXBLAW1D} reads
\begin{equation*}
\begin{aligned}
\left[\begin{aligned}
&v\\
&u\\
&Z\\
\end{aligned}\right]_t +
\left[\begin{aligned}
-u&\\
P(v,&Z) \\
0&\\
\end{aligned}\right]_x   =
\left[\begin{aligned}
\!0&\\
\!0&\\
\!-K\varphi(&\Theta(v,Z))\\
\end{aligned}\right]+\eps \Bigg(A \left[\begin{aligned}
&u\\
&v\\
&Z\\
\end{aligned}\right]_{xx}
-
 \left[\begin{aligned}
&u\\
&v\\
&Z\\
\end{aligned}\right]_{tt}\Bigg)
\end{aligned}
\end{equation*}
Clearly by \eqref{COMBCONV} hypotheses \eqref{RXENTPROP} and \eqref{AHYPSCENT1D} are satisfied with
\begin{equation*}
  \alpha = \max\Big(1, \Gamma+\bar{C}, \frac{2}{\Gamma}\bar{C}^2+2\bar{C} \Big), \quad \beta=\min\big(\frac{\gamma}{2},1\big), \quad A=2\alpha{\bf I}
\end{equation*}
while, in view of \eqref{COMPOSLIP}, the source $G$ satisfies \eqref{LISRC}. Thus, one may use Proposition \ref{Lgen-source-2}, Theorem \ref{COMP-LIPS-1D} and Theorem \ref{gen-source} to conclude about the stability, compactness and the error estimates before the formation of shocks, respectively.

\section{Acknowledgements}
A.M. thanks R.\ Young and A. Tzavaras  for several fruitful conversations during the course of this investigation.
This material is based upon work supported by the National Science
Foundation under Grant No. 0932078 000, while K.T. was in residence at
the Mathematical Sciences Research Institute in Berkeley, California, during
the  Program on ``On Optimal Transport in Geometry and Dynamics" in Fall 2013.
K.T. thanks MSRI and the organizers of the program for the hospitality, support and for providing an excellent academic environment for scientific research.  K.T.  also acknowledges the support in part by the National Science foundation under the grant DMS-1211519 and the support by the Simons Foundation under the grant $\# 267399$. K.T. thanks I.\  Kyza for several discussions in the course of the investigation and for bringing to her attention the work by Katsaounis and Makridakis \cite{KM-2003}.

\end{document}